\documentclass[12pt]{amsart}
\usepackage[top=0.75in, bottom=0.75in, left=1in, right=1in]{geometry}
\usepackage{amsmath,amssymb}
\usepackage{caption}
\usepackage{subcaption}
\usepackage{amsthm}
\numberwithin{equation}{section}
\usepackage{amsfonts}
\usepackage{graphicx}
\usepackage{hyperref}
\theoremstyle{plain}
\newtheorem{Thm}{Theorem}[section]
\newtheorem{Rem}{Remark}[section]

\newtheorem{Lem}{Lemma}[section]

\theoremstyle{definition}

\date{}
\begin{document}

\title[Self-Similar Singularity]{Self-Similar Singularity of a 1D Model for the 3D Axisymmetric Euler Equations}
\author{Thomas Y. Hou}
\author{Pengfei Liu}\thanks{California Institute of Technology, Applied and Computational Mathematics.}

\begin{abstract}
We investigate the self-similar singularity of a 1D model for the 3D axisymmetric Euler equations, which is motivated by a particular singularity formation scenario observed in numerical computation. We prove the existence of a discrete family of self-similar profiles for this model and analyze their far-field properties. The self-similar profiles we find agree with direct simulation of the model and seem to have some stability. 
\end{abstract}
\maketitle 

\section{Introduction and Main Results}
\label{intro}
Whether the 3D Euler equations develop finite-time singularity is regarded as one of the most important open problems in mathematical fluid mechanics. Interested readers may consult the survey \cite{gibbon2008three} and references therein for more historical background about this outstanding problem. In this paper we investigate the self-similar singularity of a 1D model for the 3D axisymmetric Euler equations, which is 
motivated by a particular singularity formation scenario observed in numerical computation ~\cite{luo2013potentially}. It is hoped that this work may help to understand the singularity formation of the 3D Euler equations.

In the numerical computation of Luo and Hou \cite{luo2013potentially}, the 3D axisymmetric Euler equations~ \cite{majda2002vorticity} are numerically solved in a periodic cylinder,
\begin{subequations}
	\label{3deuler}
	\begin{align}
		u_{1,t}+u^ru_{1,r}+u^zu_{1,z}&=2u_1\phi_{1,z},\\
		w_{1,t}+u^rw_{1,r}+u^zw_{1,z}&=(u_1^2)_z,\\
		-[\partial_r^2+(3/r)\partial_r+\partial_z^2]\phi_1&=w_1,\label{ebs}
	\end{align}
where $u^r=-r\phi_{1,z}$, $u^z=2\phi_1+r\phi_{1,r}$ are radial and axial velocity, and $u_1=u^\theta/r$, $w_1=w^\theta/r$, $\phi_1=\phi^\theta/r$ are transformed angular velocity, vorticity and stream function respectively.
\end{subequations}

According to the numerical results reported in \cite{luo2013potentially}, the solutions to (\ref{3deuler}) develop self-similar singularity in the meridian plane for certain initial conditions with no flow boundary condition at $r=1$. The solid boundary and special symmetry of $u^\theta$, $\omega^\theta$ and $\psi^\theta$ along the axial direction seem to make the flow in the meridian plane remain hyperbolic near the singularity point and be responsible for the observed finite-time singularity. A 1D model which approximates the dynamics of the 3D axisymmetric Euler equations along the solid boundary of the periodic cylinder $r=1$ has been proposed and investigated by Hou and Luo in ~\cite{hou2013finite}. The finite-time singularity of this model is proved very recently by Choi, Hou, Kiselev, Luo, Sverak and Yao in \cite{choi2014on}. Motivated by this new singularity formation scenario, Kiselev and Sverak ~\cite{kiselev2013small} construct an example of 2D Euler solutions in a setting similar to \cite{luo2013potentially}. In this example, the gradient of vorticity is proved to exhibit double exponential growth in time, which is known to be the fastest possible rate of growth for the 2D Euler equations. This example provides further evidence that the new singularity formation scenario reported in~\cite{luo2013potentially} is an interesting candidate to investigate the 3D Euler singularity. 

Inspired by the work of ~\cite{hou2013finite} and \cite{kiselev2013small},  Choi, Kiselev, and Yao proposed the following 1D model (we call it the CKY model for short) \cite{choi2013finite} on $[0,1]$:
\begin{subequations}
	\label{ckymodel}
\begin{align}
	\partial_tw+u\partial_xw&=\partial_{x}\rho,\label{wgrow}\\
\partial_t\rho+u\partial_x\rho&=0,\label{transportrho}\\
	u(x,t)=-x&\int_x^1\frac{w(y,t)}{y}\mathrm{d}y,\label{BS}\\
	w(0,t)=0,\quad &\rho(0,t)=0,\quad \partial_x\rho(0,t)=0.\label{bccky}
\end{align}
\end{subequations}

This 1D model, whose finite-time singularity has been proved in \cite{choi2013finite}, can be viewed as a simplified approximation to the 1D model proposed by Hou and Luo in \cite{hou2013finite}. Like the 1D model
of Hou and Luo, the CKY model approximates the 3D axisymmetric Euler equations (\ref{3deuler}) on the boundary of the cylinder $r=1$ with 
\begin{equation}
	\rho\sim u_1^2,\quad w\sim w_1,\quad u\sim u^z.
\end{equation}

The positiveness of $\rho_x(x,t)$ near the origin creates a compressive flow, which is the driving force of the finite-time singularity, and we will use this fact in our construction in section~\ref{nearfield}. Our numerical simulation suggests that this 1D model develops finite-time singularity in a way similar to that of the 3D axisymmetric Euler equations on the boundary of the cylinder as reported in \cite{luo2013potentially}, and the singular solutions to this model  also develop self-similar structure. We investigate the self-similar singularity of this CKY model in this paper.

Since the finite-time singularity of the CKY model takes place near the origin, we consider the following self-similar ansatz with $\rho$, $W$ and $U$ being the self-similar profiles,
\begin{subequations}
	\begin{align}
		\rho(x,t)&=(T-t)^{c_\rho}\rho\left(\frac{x}{(T-t)^{c_l}}\right),\label{rhoansatz}\\
		u(x,t)&=(T-t)^{c_u}U\left(\frac{x}{(T-t)^{c_l}}\right),\\
		w(x,t)&=(T-t)^{c_w}W\left(\frac{x}{(T-t)^{c_l}}\right).
	\end{align}
	\label{ansatz}
\end{subequations}

Plugging these self-similar ansatz into equations (\ref{ckymodel}) and matching the exponents of $(T-t)$ for each equation, we get
\begin{equation}
	\label{rexponent}
	c_w=-1,\quad c_u=c_l-1,\quad c_\rho=c_l-2.
\end{equation}
And the self-similar profiles $U(\xi)$, $W(\xi)$, $\rho(\xi)$ satisfy the following equations defined on $\mathbf{R}^+$, 
\begin{subequations}
	\label{targ}
	\begin{align}
		(2-c_l)\rho(\xi)+c_l\xi\rho'(\xi)+U(\xi)\rho'(\xi)&=0,\label{targrho}\\
		W(\xi)+c_l\xi W'(\xi)+U(\xi)W'(\xi)-\rho'(\xi)&=0,\label{targw}\\
		U(\xi)=-\xi\int_\xi^\infty\frac{W(\eta)}{\eta}\mathrm{d}\eta.\label{sBS}
	\end{align}
\end{subequations}

We refer equations (\ref{targ}) as the self-similar equations, which can be easily verified to enjoy the following scaling-invariant property:
\begin{equation}
	\label{rescaling}
	U(\xi)\to\frac{1}{\lambda}U(\lambda\xi),\quad W(\xi)\to W(\lambda\xi), \quad \rho(\xi)\to\frac{1}{\lambda}\rho(\lambda\xi).
\end{equation}

In this paper we study the existence and properties of solutions to the self-similar equations. A key fact for the CKY model is that the Biot-Savart law (\ref{sBS}) in the self-similar equations can be rewritten as a local relation with a decay condition,
\begin{subequations}
\begin{align}
	&\left(\frac{U(\xi)}{\xi}\right)'=\frac{W(\xi)}{\xi},\label{lbs}\\
	&\lim_{\eta\to+\infty}\frac{U(\xi)}{\xi}=0.\label{gcon}
\end{align}
\end{subequations}

We first ignore the decay condition (\ref{gcon}), then the self-similar equations with (\ref{sBS}) replaced by (\ref{lbs}) become a nonlinear ODE system with singular RHS at $\xi=0$, i.e., the RHS does not satisfy the Lipschitz condition. We construct solutions to this ODE system near $\xi=0$ using a power series method, which can naturally overcome the singularity of the RHS at $\xi=0$. The power series are unique up to a rescaling parameter for a fixed leading order of $\rho(\xi)$ at $\xi=0$, and can be extended to the whole $\mathbf{R}^+$ by solving the ODE system. Then we prove that the decay condition (\ref{gcon}) determines the scaling exponents, and there exist a discrete family of $c_l$, corresponding to different leading orders of $\rho(\xi)$, such that the decay condition (\ref{gcon}) holds for the self-similar profiles we construct . We prove this part with the assistance of numerical computation and rigorous error estimation. Given the decay condition~(\ref{gcon}), we further analyze the far-field behavior of these self-similar profiles and prove that the profiles are analytic with respect to a transformed variable at $\xi=+\infty$. 

Our main results include the following two Theorems:
\begin{Thm}
	\label{firstresult}
	There exist a discrete family of scaling exponent $c_l$ and solutions to equations~(\ref{targ}), corresponding to different leading orders of the self-similar profile $\rho(\xi)$ at $\xi=0$,
\begin{equation}
	\label{defs}
	s=\min\{k\in N^{+}|\ \frac{\mathrm{d}^k}{\mathrm{d}\xi^k}\rho(0)\neq 0\}.
\end{equation}
\end{Thm}
\begin{Thm}\label{farana}
	In the near-field, the self-similar profiles $W(\xi)$, $U(\xi)$, $\rho(\xi)$ we construct are analytic with respect to $\xi$ at $\xi=0$. And in the far-field, $W(\xi)$, $U(\xi)\xi^{-1}$, $\rho(\xi)\xi^{-1}$ are analytic with respect to a transformed variable $\theta=\xi^{-1/c_l}$ at $\theta=0$.
\end{Thm}

An interesting fact for this 1D model is that self-similar profiles (\ref{ansatz}) exist only for a discrete set of the scaling exponent $c_l$, corresponding to different leading orders of $\rho(x,t)$ at $x=0$. The self-similar profiles we find agree with direct simulation of the model and seem to have some stability in the sense that for fixed leading order of $\rho(x,0)$, the singular solutions using different initial conditions converge to the same set of self-similar profiles. 

The self-similar profiles we construct are non-conventional in the sense that the velocity does not decay to $0$ at infinity but grows with certain fractional power, correspondingly, the velocity field at the singularity time is H\"older continuous.  Such behavior is also observed in the numerical simulation of the 3D Euler equations in ~\cite{hou2013finite}, which is very different from the Leray type of self-similar solutions of the 3D Euler equations, whose existence has been ruled out under certain decay assumptions on the self-similar profiles \cite{chae2007nonexistence, chae2007nonexistence2, chae2011self}.

Our method of analysis is of interest by itself. The existence result replies on the use of a power series method to deal with the singularity of the self-similar equations at the origin, and some very subtle and relatively sharp estimates of the self-similar profiles. The same approach can be taken to analyze the self-similar singularity of Burgers equation and get results similar to those obtained in this paper. We are currently investigating the possibility of extending this method to study the singularity of the 2D Boussinesq system. 

Another novelty in our analysis is the use of numerical computation with rigorous error control, which is an important step in establishing the existence of self-similar solutions. Our strategy to rigorously control the numerical error, including the numerical error of the integration scheme for an ODE system and the roundoff error introduced due to floating point operation, is quite general and can be used for other purposes. 

The rest of this paper is organized as follows. In section~\ref{nearfield}, we construct the local self-similar profiles using a power series method and extend them to the whole $\mathbf{R}^+$. In section~\ref{matching}, we prove that the decay condition in the Biot-Savart law determines the scaling exponents in the self-similar solutions. In section~\ref{verification}, we prove the existence of self-similar profiles for different leading orders of $\rho(x,t)$ at the origin. In section~\ref{farfield}, we analyze the far-field behavior of the self-similar profiles. In section~\ref{compare}, we present our numerical results.

\section{Construction of the Near-field Solutions}
\label{nearfield}
In this section, we ignore the decay condition (\ref{gcon}), i.e. we use the local relation (\ref{lbs}) to replace the Biot-Savart law (\ref{sBS}) in the self-similar equations. We use a power series method to construct the self-similar profiles near $\xi=0$, which we call the near-field solutions. Then we prove that the local solutions constructed in this way can be extended to whole $\mathbf{R}^+$.

We have the following Theorem
\begin{Thm}\label{localpower}
	For fixed $c_l>2$, there exist a family of non-trivial analytic solutions to the self-similar equations (\ref{targ}) (with (\ref{sBS}) replaced by (\ref{lbs})) near $\xi=0$, corresponding to different leading orders of $\rho(\xi)$ at $\xi=0$, $s$, which is defined in (\ref{defs}).
\end{Thm}
\begin{proof}
Assume
\begin{subequations}
	\label{taylor}
\begin{equation}
	\rho(\xi) = \sum_{k=1}^{\infty} \rho_k\xi^k,\quad U(\xi)= \sum_{k=1}^{\infty} U_k \xi^k,\quad W(\xi)=\sum_{k=1}^{\infty}W_k\xi^k. \label{rhopower}
\end{equation}
Based on the local relation in the Biot-Savart law (\ref{lbs}), we have
\begin{equation}
	W_k=kU_{k+1}.
\end{equation}
\end{subequations}
Plugging (\ref{taylor}) into (\ref{targ}) and matching the $k$-th ($k\geq 1$) order term $\xi^k$, we get
\begin{subequations}
	\label{matchseries}
	\begin{align}
		(2-c_l)\rho_k+kc_l\rho_k+\sum_{m=1}^{k-1}(k-m+1)\rho_{k-m+1}U_m&=0,\label{rhomatch}\\
		(k-1)U_k+c_l(k-1)^2U_k+\sum_{m=1}^{k-1}U_m(k-m)^2U_{k-m+1}-k\rho_k&=0,\label{wmatch}
	\end{align}
\end{subequations}
If initially the leading order of $\rho(x,0)$ at $x=0$ is $s$, which is greater than $2$ (\ref{bccky}), then according to (\ref{transportrho}), $s$ will remain as the leading order of $\rho(x,t)$ as long as the velocity field remains smooth. And as we discussed in Section~\ref{intro}, $\rho_x^{(s)}(x,t)$ should be positive near $x=0$ to produce finite-time singularity. So in the corresponding self-similar profile~(\ref{rhopower}), there is
\begin{equation}
	\rho_i=0\ \text{for}\ i<s,\quad \rho_s>0,\ \ s\geq 2. 
	\label{rhoinitial}
\end{equation}
To make (\ref{rhomatch}) hold for $1\leq k\leq s$, we require
\begin{equation}
(2-c_l+sc_l+sU_1)\rho_s=0.
\label{du1}
\end{equation}
Since $\rho_s\neq 0$, we require
\begin{equation}
\label{u1}
U_1=\frac{(1-s)c_l-2}{s}.
\end{equation}
To make (\ref{wmatch}) hold for $1\leq k<s$, we require
\begin{equation}
[(k-1)+c_l(k-1)^2+U_1(k-1)^2]U_k=0.
\end{equation}
Since $c_l>2$, and $[(k-1)+c_l(k-1)^2+U_1(k-1)^2]>0$, we require
\begin{equation}
	U_k=0,\quad 1<k<s.
	\label{firstU}
\end{equation}
And to make (\ref{wmatch}) hold for $k=s$, we require
\begin{equation}
\label{initialu}
U_s=\frac{s^2\rho_s}{sc_l-c_l-s+2}>0.
\end{equation}
For $k>s$, to make (\ref{matchseries}) hold, the coefficients $\rho_k$ and $U_k$ should satisfy
\begin{subequations}
		\label{explicit}
	\begin{align}
		\rho_k&=\frac{-\sum_{m=s}^{k-1}U_m(k-m+1)\rho_{k-m+1}}{(k/s-1)(c_l-2)},\label{explicitrho}\\
		U_k&=\frac{k\rho_k-\sum_{m=s}^{k-1}U_m(k-m)^2U_{k-m+1}}{(k-1)+(c_l/s-2/s)(k-1)^2},\label{explicitU}
	\end{align}
\end{subequations}
which means the power series (\ref{taylor}) can be determined inductively.

To complete the proof, we still need to verify that the constructed power series (\ref{taylor}) converge for $\xi$ small enough. We choose $u^0$, $\rho^0$ and $r$ such that the following condition holds
\begin{equation}
|U_s|\leq \frac{1}{s^2}u^0r^s,\quad  |\rho_s|\leq \frac{1}{s}\rho^0r^s, \quad \frac{(s+1)u^0r}{c_l/s-2/s}\leq 1,\quad \frac{9}{4}\frac{\rho^0/u^0+u^0r}{c_l/s-2/s}<1.
	\label{initialbound}
\end{equation}
We can achieve this by choosing $u_0r$ and $\rho_0/u_0$ small enough to make the last two hold, and then choosing $r$ large enough to make the first two hold. For example, let
\begin{equation}
	A=\min\{\frac{c_l-2}{s(s+1)},\frac{2(c_l-2)}{9s}\},\ B=\frac{2(c_l-2)}{9s},\ C=\max\{\frac{s\rho_s}{AB},\frac{s^4\rho_s}{A(sc_l-c_l-s+2)}\}. 
\end{equation}
Then the choice of 
\begin{equation}
	u_0=\frac{A}{C^{1/(s-1)}},\quad \rho_0=u_0B,\quad r=C^{1/(s-1)}
\end{equation}
would satisfy (\ref{initialbound}). And we will use induction to prove that for all $k\geq s$, 
\begin{equation}
	|U_k|\leq \frac{1}{k^2}u^0r^k, |\rho_k|\leq \frac{1}{k}\rho^0r^k.
	\label{bound}
\end{equation}
For $k=s$, (\ref{bound}) holds by (\ref{initialbound}). Assume now for $s\leq k<n$, (\ref{bound}) holds. Then for $k=n\geq s+1$, based on (\ref{explicitrho}) we have
\begin{equation}
	|\rho_n|\leq \frac{\sum_{m=s}^{n-1}|U_m||(n-m+1)||\rho_{n-m+1}|}{(n-s)(c_l/s-2/s)}.
\end{equation}
Using the induction assumption and the fact that $\sum_{m=2}^\infty \frac{1}{m^2}\leq 1$, we have
\begin{equation}
	|\rho_n|\leq \frac{\rho^0u^0r^{n+1}}{(n-s)(c_l/s-2/s)}\leq \frac{\rho^0 r^n}{n}\times\frac{(s+1)u^0r}{c_l/s-2/s}\leq \frac{\rho^0 r^n}{n},
\end{equation}
where we have used the fact $n\geq s+1$ in the second inequality and (\ref{initialbound}) in the third inequality. Thus (\ref{bound}) holds for $\rho_n$. Based on (\ref{explicitU}), we have
\begin{equation}
	|U_n|\leq \frac{|n\rho_n|+\sum_{m=s}^{n-1}|U_m(n-m)^2||U_{n-m+1}|}{(c_l/s-2/s)(n-1)^2}
	\label{ubound}
\end{equation}
Using the induction assumption and the fact that $\sum_{m=2}^\infty \frac{1}{m^2}\leq 1$, we get
\begin{equation}
	|U_n|\leq \frac{\rho^0r^n+(u^0)^2r^{n+1}}{(c_l/s-2/s)(n-1)^2}\leq\frac{u_0r^n}{n^2}  \times\frac{\rho_0/u_0+u_0r}{c_l/s-2/s}\times\frac{n^2}{(n-1)^2}\leq \frac{u_0r^n}{n^2},
\end{equation}
where we have used (\ref{initialbound}) and the fact that 
$n\geq 3$, $n^2/(n-1)^2\leq 9/4$ in the last inequality. \\

So we just proved (\ref{bound}) by induction, which implies the power series (\ref{taylor}) converge in some short interval $[0, 1/r)$. This completes the proof of Theorem~\ref{localpower}.
\end{proof}
\begin{Rem}\label{cl}
	We require $c_l>2$ in Theorem \ref{localpower}. If $c_l=2$, there exist only trivial solutions to (\ref{targ}). If $c_l<2$, then $c_\rho<0$ according to (\ref{rexponent}), which means $\rho(x,t)$ blows up in finite time according to (\ref{ansatz}). This is impossible since $\rho(x,t)$ is transported by the fluid. 
\end{Rem}
\begin{Rem}\label{dof}
	For fixed $c_l$, we have one degree of freedom $\rho_s>0$ in constructing the near-field solutions (\ref{taylor}), which can be easily verified to play the role of a scaling parameter (\ref{rescaling}). We will simply choose $\rho_s=1$ in our argument for the rest part of this paper.
\end{Rem}

The power series (\ref{taylor}) we construct only converge in a short interval near $\xi=0$. However, these local self-similar profiles can be extended to $+\infty$. 

\begin{Thm}\label{gexist}
	For $c_l>2$, the analytic solutions (\ref{taylor}) we construct can be extended to the whole $\mathbf{R}^+$, resulting in solutions to the self-similar equations (\ref{targ}) with (\ref{sBS}) replaced by (\ref{lbs}). Moreover, we have for $\xi>0$,
\begin{equation}
\label{basicprofile}
W(\xi)>0,\quad \rho(\xi)>0.
\end{equation}
\end{Thm}
\begin{proof}
	Since $c_l+U_1=(c_l-2)/s>0$, $\rho_s>0$, $W_s=(s-1)U_s>0$, based on the leading orders of the power series (\ref{taylor}), we can choose $\epsilon<\frac{1}{r}$ small enough such that 
	\begin{equation}
		c_l\epsilon+U(\epsilon)>0,\quad W(\epsilon)>0,\quad \rho(\epsilon)>0.
	\end{equation}
	
	Then we consider extending the self-similar profiles from $\xi=\epsilon$ to $+\infty$ by solving an ODE system with initial conditions given by the power series (\ref{taylor}). Let $\tilde{U}(\xi)=c_l\xi+U(\xi)$, then according to (\ref{targ}), $\tilde U(\xi)$, $\rho(\xi)$ and $U(\xi)$ satisfy the following ODE
	\begin{subequations}
		\begin{align}
			\rho'(\xi)&=\frac{(c_l-2)\rho(\xi)}{\tilde U(\xi)},\label{odes1}\\
			W'(\xi)&=\frac{(c_l-2)\rho(\xi)}{\tilde U(\xi)^2}-\frac{W(\xi)}{\tilde U(\xi)},\label{odes2}\\
			(\frac{\tilde U(\xi)}{\xi})'&=\frac{W(\xi)}{\xi}\label{odes3}.
		\end{align}
		\label{odes}
	\end{subequations}

The right hand side of (\ref{odes}) is locally Lipschitz continuous for $\tilde U(\xi)\neq 0$, $\xi\neq 0$, so we can solve the ODE system from $\epsilon$ and get its solution on interval $[\epsilon, T)$. We first prove that $W(\xi)$ is positive on $[\epsilon, T)$. Otherwise denote $\xi=t$ as the first time $W(\xi)$ reaches $0$, i.e.
	\begin{equation}
		t=\inf\{s\in[\epsilon, T): W(s)\leq 0\}.
		\end{equation}
		Then we have $W(\xi)$ is positive on $[\epsilon, t)$, and 
	\begin{equation}
		\label{fakeassumption}
		W'(t)\leq 0.
	\end{equation}
Based on (\ref{odes3}), $\frac{\tilde U(\xi)}{\xi}$ is increasing on $[\epsilon, t)$, thus $\tilde U(\xi)> \tilde U(\epsilon)>0$ for $\xi\in[\epsilon, t]$. Then based on (\ref{odes1}), $\rho(\xi)$ is increasing on $[\epsilon, t]$, and $\rho(t)>0$. Evaluating (\ref{odes2}) at $\xi=t$, we get 
\begin{equation}
	W'(t)=\frac{(c_l-2)\rho(t)}{\tilde U(t)^2}>0,
\end{equation}
which contradicts with (\ref{fakeassumption}). So $W(\xi)>0$ and consequently $\rho(\xi)>0$ for $\xi\in [\epsilon, T)$.

Using the fact that $W(\xi)>0$ in (\ref{odes3}), we have for $\xi>\epsilon$,
\begin{equation}
	\tilde U(\xi)\geq C_0\xi.
	\label{ulow}
\end{equation}
Using this lower bound in (\ref{odes1}), we get
\begin{equation}
	\rho'(\xi)\leq \frac{C_1\rho(\xi)}{\xi}.
\end{equation}
This implies that for $\xi>\epsilon$
\begin{equation}
	\rho(\xi)\leq C_2\xi^{C_1}.
	\label{rhoupp}
\end{equation}
Using (\ref{rhoupp}), (\ref{ulow}) and the fact that $W(\xi)$ is positive in (\ref{odes2}), we have 
\begin{equation}
	W'(\xi)\leq C_3\xi^{C_1-2}.
\end{equation}
Thus for $\xi>\epsilon$,
\begin{equation}
	W(\xi)\leq C_4\xi^{C_1}.
	\label{Wupp}
\end{equation}
Finally using (\ref{Wupp}) in (\ref{odes3}), we get for $\xi>\epsilon$,
\begin{equation}
	U(\xi)\leq C_{5}\xi^{C_{1}+2}.
	\label{Uupp}
\end{equation}
The $C_0$, $C_1$,\dots $C_{5}$ in the above estimates are positive constants. These {\em a priori} estimates (\ref{ulow}), (\ref{Uupp}), (\ref{rhoupp}) and (\ref{Wupp}) together imply that we can get solutions to (\ref{odes}) on $[\epsilon, +\infty)$, i.e. the local self-similar profile constructed using power series can be extended to $+\infty$.
\end{proof}

\section{Determination of the Scaling Exponents}
\label{matching}
In our construction of self-similar profiles in the previous section, we did not consider the decay condition (\ref{gcon}). In this section, we prove that the decay condition determines the scaling exponent $c_l$, i.e. only for certain $c_l$ does the decay condition hold.

Recall that for a fixed leading order of $\rho(\xi)$, the self-similar profiles $U(\xi)$, $\rho(\xi)$ and $W(\xi)$ depend on the scaling exponent $c_l$ only. So we can define a function $G(c_l)$ as 
\begin{equation}
	G(c_l)=\lim_{\xi\to+\infty}\frac{U(\xi)}{\xi}.	
\end{equation}
We will prove that $G(c_l)<+\infty$ and it is a continuous function of $c_l$. Then the existence of $c_l$ to satisfy the decay condition (\ref{gcon}) will follow from the Intermediate Value Theorem if we can show that there exist $c_l^l$ and $c_l^r$ such that 
\begin{equation}
	\label{bisection}
	G(c_l^l)<0,\quad  G(c_l^r)>0.
\end{equation}

\begin{Thm}\label{exist}
	For a fixed leading order of $\rho(\xi)$, $s$, and a scaling exponent $c_l>2$, construct the power series (\ref{taylor}) with $\rho_s=1$ and extend the profiles to $\mathbf{R}^+$. Then the limit 
	\begin{equation}
		G(c_l)=\lim_{\xi\to\infty}\frac{U(\xi)}{\xi}<+\infty,
	\end{equation}
	and $G(c_l)$ is a continuous function of $c_l$.
\end{Thm}

We first make the following change of variables, 
\begin{equation}
	\label{changeeta}
	\eta=\xi^{1/c_l},\quad \hat W(\eta)=W(\xi),\quad \hat U(\eta)=U(\xi)\xi^{-1},\quad \hat \rho(\eta)=\rho(\xi)\xi^{-1+2/c_l}.
\end{equation}
Then we have 
\begin{equation}
	G(c_l)=\lim_{\eta\to+\infty}\hat U(\eta),
\end{equation}
and the ODE system satisfied by $\hat U(\eta), \hat \rho(\eta), \hat W(\eta)$ is
\begin{subequations}
	\begin{align}
		\hat\rho'(\eta)&=\frac{(2/c_l-1)\hat\rho(\eta)\hat U(\eta)}{\eta+1/c_l\hat U(\eta)\eta},\label{farode1}\\
		\hat W'(\eta)&=\frac{-\hat W(\eta)}{\eta+1/c_l\hat U(\eta)\eta}+\frac{(1-2/c_l)\hat \rho(\eta)}{(1+1/c_l\hat U(\eta))^2\eta^3},\label{farode2}\\
		\hat U'(\eta)&=\frac{c_l\hat W(\eta)}{\eta}.\label{farode3}
	\end{align}	
	\label{farode}
\end{subequations}
According to (\ref{u1}), (\ref{basicprofile}) and the fact that $\hat U(\eta)$ is monotone increasing, we have 
\begin{equation}
	\label{prepare}
	\hat U(\eta)>\hat U(0)=\frac{(1-s)c_l-2}{s}, \quad \hat W(\eta)>0,\quad \hat \rho(\eta)>0, \quad\text{for}\quad \eta>0.
\end{equation}
Before proving Theorem~{\ref{exist}}, we will first prove the following two Lemmas.
\begin{Lem}\label{lem2}
	For all $c_l>2$, $G(c_l)>-2$.
\end{Lem}
\begin{proof}
	Assume that for some $c_l>2$, $G(c_l)\leq -2$. Then according to (\ref{prepare}) and the fact that $\hat U(\eta)$ is increasing, we have for all $\eta>0$,
\begin{equation}
	\frac{(1-s)c_l-2}{s}<\hat U(\eta)<-2.
\end{equation}
Then we get
\begin{equation}
	\frac{(2/c_l-1)\hat U(\eta)}{1+1/c_l\hat U(\eta)}\geq 2.
\end{equation}
It follows from (\ref{farode1}) that
	\begin{equation}
		\hat \rho'(\eta)\geq 2\frac{\hat \rho(\eta)}{\eta}.
	\end{equation}
By direct integration and (\ref{prepare}), we have for $\eta$ large enough,
\begin{equation}
		\hat \rho(\eta)\geq C_1\eta^2.
	\end{equation}
Using this estimate in (\ref{farode2}), we get
\begin{equation}
	\hat W'(\eta)\geq-\frac{C_2\hat W(\eta)}{\eta}+\frac{C_3}{\eta}.
\end{equation}
This implies
\begin{equation}
	\big(\eta^{C_2}\hat W(\eta)\big)'\geq C_3\eta^{C_2-1}.
\end{equation}
Then we have for $\eta$ large enough,
\begin{equation}
	\eta^{C_2}\hat W(\eta)\geq  \frac{C_3}{C_2}\eta^{C_2}-C_4.
\end{equation}
Using this lower bound in (\ref{farode3}), we will get 
\begin{equation}
	\label{Ulowerlemma}
	\hat U'(\eta)\geq\frac{C_5}{\eta}-\frac{C_6}{\eta^{C_2+2}}.
\end{equation}
The constants $C$ in the above estimates are positive and independent of $\eta$. 
The inequality
(\ref{Ulowerlemma}) implies that $\hat U(\eta)\to+\infty$ as $\eta\to+\infty$, which contradicts with $G(c_l)\leq -2$. This completes the proof of Lemma \ref{lem2}.
\end{proof}
We add a subscript $c_l$ to indicate the dependence of the profiles on $c_l$ for the rest part of this section:
\begin{equation}
	\hat U_{c_l}(\eta)=\hat U(\eta),\ \hat W_{c_l}(\eta)=\hat W(\eta),\ \hat W_{c_l}(\eta)=\hat W(\eta).
\end{equation}

\begin{Lem}
	\label{lem3}
	For fixed $\eta>0$, $\hat U_{c_l}(\eta)$, $\hat W_{c_l}(\eta)$ and $\hat \rho_{c_l}(\eta)$ are continuous functions of $c_l$.
\end{Lem}
\begin{proof}
	We only need to prove that for fixed $c_l^0>2$, $\hat U_{c_l}(\eta)$, $\hat \rho_{c_l}(\eta)$ and $\hat W_{c_l}(\eta)$ as functions of $c_l$ are continuous at $c_l=c_l^0$. In our construction of the power series using (\ref{explicit}), we can easily see that the coefficients $U_k$ and $\rho_k$ depend continuously on $c_l$. And based on the condition~(\ref{initialbound}), there exist uniform upper bounds of these coefficients
	\begin{equation}
	|U_k|\leq \frac{u^0r^k}{k^2},\quad |\rho_k|\leq \frac{\rho^0r^k}{k},
	\end{equation}
	for $c_l$ in a neighbourhood of $c_l^0$. 
	This means there exists a fixed $\epsilon$ small enough, such that $\hat W_{c_l}(\epsilon)$, $\hat \rho_{c_l}(\epsilon)$ and $\hat U_{c_l}(\epsilon)$ are continuous at $c_l^0$. 
	Then we use the continuous dependence of ODE solutions on initial conditions and parameter to complete the proof of this Lemma. 
\end{proof}
Now we begin to prove Theorem~\ref{exist}. We use an iterative method which enables to get shaper estimates of the profiles after each iteration. We finally attain that $\hat U_{c_l}(\eta)$ converges uniformly to $G(c_l)$, with which we can complete the proof of this Theorem.
\begin{proof}
	Consider $c_l^0>2$, we will prove that $G(c_l^0)< +\infty$, and $G(c_l)$ is continuous at $c_l=c_l^0$.

	According to Lemma \ref{lem2} and Lemma \ref{lem3}, there exists $\eta_0$ large enough and a neighborhood of $c_l^0$, $I_0=(c_1,c_2)$ with $c_1>2, c_2<+\infty$ such that for $c_l\in I_0$ and $\eta>\eta_0$, 
\begin{equation}
	\label{farulower}
	\hat U_{c_l}(\eta)>\hat U_{c_l}(\eta_0)>-2+\epsilon_1.
\end{equation}
Then for $c_l\in I_0$ and $\eta>\eta_0$, there exist $\epsilon_2>0$, such that
	\begin{equation}
		\frac{(2/c_l-1)\hat U_{c_l}(\eta)}{1+1/c_l\hat U_{c_l}(\eta)}<2-\epsilon_2.
	\end{equation}
	Using this in (\ref{farode1}), we have for $c_l\in I_0$ and $\eta>\eta_0$, 
	\begin{equation}
		\hat\rho_{c_l}'(\eta)\leq \frac{(2-\epsilon_2)\hat \rho_{c_l}(\eta)}{\eta}.
	\end{equation}
	Using direct integration and Lemma \ref{lem3}, we have for $c_l\in I_0$, $\eta>\eta_0$,
	\begin{equation}
		\hat\rho_{c_l}(\eta)\leq C_1\eta^{2-\epsilon_2}.
	\end{equation}
	Using this upper bound of $\hat \rho(\eta)$ in (\ref{farode2}), we have for $c_l\in I_0$, $\eta>\eta_0$,
	\begin{equation}
		\label{wgrowth}
		\hat W_{c_l}'(\eta)\leq \left(\frac{-1}{1+1/c_l\hat U_{c_l}(\eta)}\right)\frac{\hat W_{c_l}(\eta)}{\eta}+C_3\eta^{-1-\epsilon_2}.
	\end{equation}
	The first term in (\ref{wgrowth}) is negative according to (\ref{prepare}) and the second term is integrable for $\eta>\eta_0$. Then using Lemma~\ref{lem3}, we have for $c_l\in I_0$, $\eta>\eta_0$,
	\begin{equation}
		\hat W_{c_l}(\eta)<C_4	
	\end{equation}
	Putting this upper bound in (\ref{farode3}) and using Lemma~\ref{lem3}, we get for $c_l\in I_0$, $\eta>\eta_0$,	
	\begin{equation}
		\hat U_{c_l}(\eta)<C_5\ln\eta.
\end{equation}
Putting this upper bound of $\hat U(\eta)$ back in (\ref{farode2}), we have for $c_l\in I_0$, $\eta>\eta_0$
\begin{equation}
	\hat W_{c_l}'(\eta)<-\frac{C_6\hat W_{c_l}(\eta)}{\eta\ln\eta}+C_3\eta^{-1-\epsilon_2},
\end{equation}
which by direct integration gives that for $c_l\in I_0$, $\eta>\eta_0$,
\begin{equation}
	\hat W_{c_l}(\eta)\exp(\int_{\eta_0}^\eta \frac{C_6}{\zeta\ln \zeta}\mathrm{d}\zeta)<C_7.
\end{equation}
Thus we have for $c_l\in I_0$ and $\eta>\eta_0$,  
\begin{equation}
	\hat W_{c_l}(\eta)<C_8/\ln\eta.
\end{equation}
Using this sharper upper bound of $W(\eta)$ in (\ref{farode3}), we get for $c_l\in I_0$, $\eta>\eta_0$,
\begin{equation}
	\hat U_{c_l}(\eta)<C_9\ln\ln\eta.
\end{equation}
Again putting this sharper upper bound in (\ref{farode2}), we have for $c_l\in I_0$, $\eta>\eta_0$,
\begin{equation}
	\hat W_{c_l}'(\eta)<-\frac{C_{10}\hat W_{c_l}(\eta)}{\eta\ln\eta}+C_3\eta^{-1-\epsilon_2},
\end{equation}
By direct integration, we get
\begin{equation}
	\hat W_{c_l}(\eta)\exp(\int_{\eta_0}^\eta \frac{C_{11}}{\zeta\ln\ln\zeta}\mathrm{d}\zeta)<C_{12}.
\end{equation}
Since $\int_{\eta_0}^\eta \frac{C_{11}}{\zeta\ln\ln\zeta}\mathrm{d}\zeta>C_{13}(\ln \eta)^\alpha$ for some $\alpha\in(0,1)$, we have for $c_l\in I_0$, $\eta>\eta_0$, 
	\begin{equation}
		\label{finalupperw}
		\hat W_{c_l}(\eta_1)<C_{14}\exp\big(-C_{13}(\ln \eta)^\alpha\big).
	\end{equation}
	Note that $C_1, C_2,\dots C_{14}$ in the above estimates are all positive constants in dependent of $\eta$. Using the upper bound of $\hat W_{c_l}(\eta)$ (\ref{finalupperw}) in (\ref{farode3}) we conclude that $\hat U_{c_l}(\eta)$ converges uniformly as $\eta\to +\infty$ for $c_l\in I_0$ and complete the proof of this Theorem.
\end{proof}
To complete the proof of our main result Theorem~\ref{firstresult}, we still need to verify condition~(\ref{bisection}) for different $s$. And we leave this part to section~\ref{verification}.

\section{Existence of Self-Similar Profiles}
\label{verification}
In this section, we verify that condition (\ref{bisection}) holds for several $s\geq 2$, i.e., there exist $c_l^l$ and $c_l^r>2$, such that $G(c_l^l)<0,\ G(c_l^r)>0$, with which we can complete the proof of Theorem~\ref{firstresult}.
We will use the following Lemma, which allows us to prove (\ref{bisection}) using only estimates of the profiles at some fixed $\eta_0$. 
\begin{Lem}
\begin{subequations}
	Consider solving equations (\ref{farode}) with initial conditions given by power series~(\ref{taylor}). For some $\eta_0>0$, let $u_0=\hat U(\eta_0)$, $\rho_0=\hat \rho(\eta_0)$, $w_0=\hat W(\eta_0)$.\\
If
	\begin{equation}
	\label{assumpp}
	u_0>0,
	\end{equation}
	then we have
	\begin{equation}
		\label{firstpart}
	G(c_l)>0.
	\end{equation}
	If 
	\begin{equation}
		\label{assump}
		u_0>-2, \quad u_0+c_lw_0+\frac{(c_l-2)\rho_0}{(u_0+2)(1+u_0/c_l)}<0, 
	\end{equation}
Then 
\begin{equation}
	G(c_l)<0.
\end{equation}
	\end{subequations}
\end{Lem}
\begin{proof}
	Since $G(c_l)=\lim_{\eta\to+\infty}\hat U(\eta)$, and $\hat U(\eta)$ is increasing according to (\ref{farode3}) and (\ref{basicprofile}), so if $u_0>0$, then $G(c_l)>u_0>0$, and we finish the first part of the Lemma (\ref{firstpart}). 
	
	We prove the second part by contradiction. If $G(c_l)\geq 0$, there exists $\eta_1\in(\eta_0,+\infty]$ such that $\hat U(\eta_1)=0$. Then for $\eta\in(\eta_0,\eta_1)$, $\hat U(\eta)>u_0$, and according to (\ref{farode1}) we have,
	\begin{subequations}
	\begin{equation}
		\hat \rho'(\eta)\leq\frac{(2/c_l-1)u_0}{1+u_0/c_l}\frac{\hat\rho(\eta)}{\eta}.
	\end{equation}
By direct integration, we get for $\eta\in(\eta_0,\eta_1)$,
	\begin{equation}
		\hat \rho(\eta)\leq \rho^0\eta_0^{\frac{(1-2/c_l)u_0}{1+u_0/c_l}}\eta^{\frac{(2/c_l-1)u_0}{1+u_0/c_l}}.
	\end{equation}
	\end{subequations}
	Using this upper bound of $\hat \rho$ and the fact that $\hat U(\eta)<0$ for $\eta\in(\eta_0,\eta_1)$ in (\ref{farode2}), we get
	\begin{subequations}
	\begin{equation}
		\label{bweta}
		(\hat W(\eta)\eta)'\leq \frac{1-2/c_l}{(1+u_0/c_l)^2}\rho_0\eta_0^{\frac{(1-2/c_l)u_0}{1+u_0/c_l}}\eta^{\frac{-u_0-2}{1+u_0/c_l}}.
	\end{equation}
	Since $u_0>-2$, integrating (\ref{bweta}) from $\eta_0$ to $\eta$, we have for $\eta\in (\eta_0,\eta_1)$
	\begin{equation}
		\hat W(\eta)\eta\leq w_0\eta_0+\frac{2/c_l-1}{(1+u_0/c_l)(u_0/c_l-u_0-1)}\rho_0(\eta_0^{-1}-\eta_0^{\frac{(1-2/c_l)u_0}{1+u_0/c_l}}\eta^{\frac{-u_0-1+u_0/c_l}{1+u_0/c_l}}).
	\end{equation}
	\end{subequations}
Putting this upper bound of $\hat W(\eta)$ in (\ref{farode3}) and integrating it from $\eta_0$ to $\eta_1$, we get
\begin{equation}
0-u_0=\hat U(\eta_1)-\hat U(\eta_0)\leq c_lw_0+\frac{(c_l-2)\rho_0}{(u_0+2)(1+u_0/c_l)},
\end{equation}
which contradicts (\ref{assump}), and we complete the proof of this Lemma.
\end{proof}

We will use numerical computation and rigorous error estimation to verify condition (\ref{assumpp}) or (\ref{assump}). We first numerically construct the power series (\ref{taylor}) and then extend the local self-similar profiles to some $\eta_0$ by numerically solving an ODE system. 

Note that in equations (\ref{farode}), there are $1/\eta^3$ terms in the right hand side, which can be very large for small $\eta$. This will make the numerical solutions sensitive to roundoff error. So instead of solving (\ref{farode}) directly, we make the following change of variables,
\begin{equation}
	\label{solvechange}
	\tilde \rho(\eta)=\hat \rho(\eta) \eta^{-5},\quad \tilde W(\eta)=\hat W(\eta)\eta^{-1},\quad \tilde U(\eta)=\hat U(\eta),
\end{equation}
and the equations satisfied by this new set of unknowns are given by
\begin{subequations}
	\begin{align}
		\tilde \rho'(\eta)&=\frac{(-3-c_l)\tilde \rho(\eta)\tilde U(\eta)-5c_l\tilde \rho(\eta)}{c_l\eta+\tilde U(\eta)\eta},\label{solvefar1}\\
		\tilde W'(\eta)&=\frac{-2c_l\tilde W(\eta)-\tilde U(\eta)\tilde W(\eta)}{c_l\eta+\tilde U(\eta)\eta}+\frac{c_l(c_l-2)\tilde \rho(\eta)\eta}{(c_l+\tilde U(\eta))^2},\label{solvefar2}\\
		\tilde U'(\eta)&=c_l\tilde W(\eta).\label{solvefar3}
	\end{align}
After this change of variables, there are only $1/\eta$ terms for this new system of ODEs.
	\label{solvefar}
\end{subequations}

We demonstrate how to rigorously bound the numerical error of the self-similar profiles at $\eta_0$ through the case $s=2$, but the same procedure can be applied to other $s$ to prove the existence of self-similar profiles. For the case $s=2$, we choose $c_l^l=3$ and $c_l^r=8$ in (\ref{bisection}).
\subsection{The case $s=2$, $c_l=3$} We need to go though the following several steps.

\noindent
{\em Step 1}\quad We need to bound the truncation error of the power series (\ref{taylor}). To numerically compute the power series (\ref{taylor}), we first truncate the power series to certain terms, and the truncation error can be bounded using (\ref{bound}). For the case $s=2$, $c_l=3$, it can be easily verified that the following choice of $\rho^0$, $u^0$ and $r$ makes (\ref{initialbound}) hold:
\begin{equation}
	u^0=\frac{1}{9\times 162},\quad \rho^0=\frac{1}{9\times 9\times 162},\quad r=162.
\end{equation}
Based on (\ref{solvechange}) and (\ref{taylor}), at $\xi=10^{-3}$, corresponding to $\eta_s=10^{-1}$, we have 
\begin{equation}
	\label{initialprofile}
	\tilde U(\eta_s)=\sum_{k=1}^{\infty}U_k\eta^{3k-3},\quad \tilde \rho(\eta_s)=\sum_{k=2}^{\infty}\rho_k\eta^{3k-6},\quad \tilde W(\eta_s)=\sum_{k=1}^{\infty}W_k\eta^{3k-1}.
\end{equation}
Using estimate (\ref{bound}), if we truncate the series (\ref{initialprofile}) at $k=20$, the truncation error for all the three series can be bounded by $10^{-15}$. We will numerically compute 
\begin{equation}
	\label{partialsum}
	\tilde U(\eta_s)\approx\sum_{k=1}^{20}U_k\eta^{3k-3},\quad \tilde \rho(\eta_s)\approx\sum_{k=2}^{20}\rho_k\eta^{3k-6},\quad \tilde W(\eta_s)\approx\sum_{k=1}^{20}W_k\eta^{3k-1}.
\end{equation}
and use them as initial conditions to solve (\ref{solvefar}).\\

\noindent
{\em Step 2}\quad We need to bound the roundoff error in computing the truncated power series (\ref{partialsum}). Denote $fl$ as the floating point operation and assume $a$ and $b$ are two floating point numbers, which can be stored exactly on computer. Then by the IEEE standard rounding off rules~ \cite{2008ieee}, we have if $fl(a\odot b)\neq 0$,
\begin{equation}
\label{roundingrule}
\mathrm{fl}(a\odot b)=(a\odot b)(1+\delta), \quad |\delta|\leq \epsilon,
\end{equation}
where $\odot$ can be $+$, $-$, $\times$ and $\div$, and $\epsilon$ is the machine precision. When using double precision floating point operation on matlab, there is
\begin{equation}
\epsilon<3\times 10^{-15}.
\end{equation}

The following two Lemmas will be intensively used to control the roundoff error.
\begin{Lem}
\label{boundroundoff}
Assuming $a$ and $b$ are two floating point numbers stored in the computer, then we have the following upper and lower bounds,\\
\begin{subequations}
	\label{roundlemma}
if $\mathrm{fl}(a\odot b)>0$,
\begin{align}
&a\odot b\leq \mathrm{fl}\Big(\mathrm{fl}(a\odot b)\times \mathrm{fl}\big(1+\mathrm{fl}(4\epsilon)\big)\Big)\label{flupper},\quad a\odot b\geq \mathrm{fl}\Big(\mathrm{fl}(a\odot b)\times \mathrm{fl}\big(1-\mathrm{fl}(4\epsilon)\big)\Big),
\end{align}
if $\mathrm{fl}(a\odot b)<0$,
\begin{align}
&a\odot b\leq \mathrm{fl}\Big(\mathrm{fl}(a\odot b)\times \mathrm{fl}\big(1-\mathrm{fl}(4\epsilon)\big)\Big),\quad a\odot b\geq \mathrm{fl}\Big(\mathrm{fl}(a\odot b)\times \mathrm{fl}\big(1+\mathrm{fl}(4\epsilon)\big)\Big).
\end{align}
\end{subequations}
\end{Lem}
The RHS of the inequaltities (\ref{roundlemma}) only involve floating point operation, which implies that even though we cannot get the exact answer because of the roundoff error, we can get rigorous lower and upper bounds of the answer using only floating point operation. The basic idea in the above lower and upper bounds is compensating the roundoff error by multiplying $(1\pm n\epsilon)$. Since $\epsilon$ is  very small, these lower and upper bounds are actually very tight. We only prove the upper bound in (\ref{flupper}) for the case $\odot$ is `$+$'. Others are similar.
\begin{proof}
According to the rounding rule (\ref{roundingrule}),
\begin{align}
\mathrm{fl}\Big(\mathrm{fl}(a+b)\times \mathrm{fl}\big(1+\mathrm{fl}(4\epsilon)\big)\Big)&=\mathrm{fl}\Big((a+b)(1+\delta_1)\times \mathrm{fl}\big(1+4\epsilon(1+\delta_2)\big)\Big)\\
							&=\mathrm{fl}\Big( (a+b)(1+\delta_1)\times (1+\delta_3)\big(1+4\epsilon(1+\delta_2)\big)\Big)\\
							&=(a+b)(1+\delta_1)(1+\delta_3)(1+\delta_4)(1+4\epsilon+4\epsilon\delta_2)\\
							&\geq (a+b).
\end{align}
In the last inequality we have used $|\delta_1|$, $|\delta_2|$, $|\delta_3|$, $|\delta_4|\leq\epsilon$.
\end{proof}
\begin{Lem}
	\label{bmonotone}
	Assuming we have lower and upper bounds of $a$ and $b$,
	\begin{equation}
		a^{\max},\quad a^{\min},\quad b^{\max},\quad b^{\min}.
	\end{equation}
Let
	\begin{equation}
		\label{set}
		V=\{a^{\max}\odot b^{\max}, a^{\max}\odot b^{\min}, a^{\min}\odot b^{\max}, a^{\min}\odot b^{\min}\},
	\end{equation}
then
\begin{equation}
	\label{infsupv}
	a\odot b\leq \sup V,\quad a\odot b\geq \inf V,
\end{equation}
where $\odot$ can be $+$, $-$, $\times$ and $\div$. When $\odot$ is $\div$, we require $0 \notin [b^{\min},b^{\max}]$.
\end{Lem}
Lemma~\ref{bmonotone} is obvious, so we omit its proof here. 

\begin{Rem}
Lemma~\ref{roundlemma} and Lemma~\ref{bmonotone} allow us to rigorously control the roundoff error introduced in our numerical computation: Assuming we have lower and upper bounds of $a$ and $b$, which are floating point numbers stored exactly on computer. Then we can use Lemma~\ref{boundroundoff} to get lower and upper bounds for each element of (\ref{set}). Then putting these bounds together in (\ref{infsupv}) and using Lemma~\ref{bmonotone}, we can get rigorous bounds of $a\odot b$. The numerical computations involved in this paper are all composed of these 4 basic arithmetic operations, so by using the two Lemmas in each single step of our computation, we can rigorously bound the roundoff error in our final results.
\end{Rem}

When computing the truncated power series (\ref{partialsum}), we need to first compute the coefficients $U_k$ and $\rho_k$ based on (\ref{explicit}). The relation (\ref{explicit}), Lemma~\ref{bmonotone} and Lemma~\ref{boundroundoff} together allow us to get lower and upper bounds of the coefficients $U_k$ and $\rho_k$ inductively: Assuming now we have got lower and upper bounds of $U_m$ and $\rho_m$ for $m<k$, which are all floating point numbers, then based on (\ref{explicitrho}), and using Lemma~\ref{boundroundoff} and Lemma~\ref{bmonotone} in each single step of the computation, we can get lower and upper bounds of $\rho_k$. Then we use the same technique in (\ref{explicitU}) to get lower and upper bounds of $U_k$.

After getting the upper and lower bounds of $\rho_k$ and $U_k$ for all $k\leq 20$, we use Lemma~\ref{bmonotone} and Lemma~\ref{boundroundoff} in computing (\ref{partialsum}), and get lower and upper bounds of the truncated power series. Based on our computation, the difference of the lower bounds and upper bounds are all less than $5\times10^{-14}$, which means if we use the upper bounds to approximate the real values of the truncated power series (\ref{partialsum}), the numerical errors are less than $5\times10^{-14}$. We denote the upper bounds of the truncated power series (\ref{partialsum}) as
\begin{equation}
\label{farinitial}
\tilde y_0=(\tilde w_0,\tilde u_0,\tilde \rho_0)^T.
\end{equation}
Since the truncation error of the power series are less than $10^{-15}$, we get a rigorous error bound of the self-similar profiles (\ref{farinitial}) at $\eta_s=0.1$, which we denote by $E_0=(E_0^w,E_0^u,E_0^\rho)^T$,
\begin{equation}
\label{farinitialbound}
E_0=(10^{-13},10^{-13},10^{-13})^T.
\end{equation}
We will use (\ref{farinitial}) and (\ref{farinitialbound}) as initial conditions to numerically solve (\ref{solvefar}).\\

\noindent
{\em Step 3}\quad We need to bound the numerical error introduced in numerically solving (\ref{solvefar}). We will use the forward Euler method with step size $h=2.9\times 10^{-6}$ to solve (\ref{solvefar}) from $\eta_s=10^{-1}$ to $\eta_0=3$. We denote the node point and numerical solutions at the $n$-th step as
\begin{equation}
x_n=0.1+nh, \quad \tilde y_n=(\tilde u_n, \tilde \rho_n, \tilde w_n)^T,\quad n=0,\dots, 10^6,
\end{equation}
and the exact solutions and error bounds at the $n$-th step as
\begin{equation}
	y_n=(u_n,\rho_n, w_n)^T,\quad E_n=(E^w_n, E^u_n, E^\rho_n)^T,\quad n=0,\dots 10^6,
\end{equation}
which means
\begin{equation}
|y_n-\tilde y_n|\leq E^n.
\end{equation}

Based on the previous step, $\tilde y_0$ and $E^0$ given by (\ref{farinitial}) and (\ref{farinitialbound}). Our approach to bound the error introduced in numerically solving (\ref{solvefar}) is simultaneously tracking $\tilde y_n$ and $E^n$, i.e. in each step of the forward Euler method we update $\tilde y_n$ and $E^n$ together,
\begin{equation}
(\tilde y_n, E^n)\to (\tilde y_{n+1},E^{n+1}).
\end{equation}

Denoting $f=(f_w,f_u,f_\rho)^T$ as the RHS of (\ref{solvefar}), the forward Euler method gives
\begin{subequations}
	\begin{align}
	\label{numericalsolu}
\tilde w_{n+1}&=\tilde w_n+f_w(\tilde w_n, \tilde u_n, \tilde \rho_n, x_n)h+e^w_{\mathrm{roundoff}},\\
\tilde u_{n+1}&=\tilde u_n+f_u(\tilde w_n, \tilde u_n, \tilde \rho_n)h+e^u_{\mathrm{roundoff}},\\
\tilde \rho_{n+1}&=\tilde \rho_n+f_\rho(\tilde w_n, \tilde u_n, \tilde \rho_n)h+e^\rho_{\mathrm{roundoff}}.
\end{align}
\end{subequations}

For the exact solutions at $x_n$ and $x_{n+1}$ which are $y_{n}$ and $y_{n+1}$, we have
\begin{subequations}
\label{realsolu}
	\begin{align}
		w_{n+1}&=w_n+f_w(w_n, u_n, \rho_n, x_n)h+1/2\tilde W''(w^1, u^1, \rho^1,x^1)h^2,\\
		u_{n+1}&=u_n+f_u(w_n, u_n, \rho_n, x_n)h+1/2\tilde U''(w^2, u^2, \rho^2,x^2)h^2,\\
		\rho_{n+1}&=\rho_n+f_\rho(w_n, u_n, \rho_n, x_n)h+1/2\tilde \rho''(w^3, u^3, \rho^3,x^2)h^2,
	\end{align}
	where $x^i\in(x_n,x_{n+1})$, $w^i=\tilde W(x^i)$, $\rho^i=\tilde \rho(x^i)$, $u^i=\tilde U(x^i)$, $i=1,2,3$.
\end{subequations}

Deducting (\ref{numericalsolu}) from (\ref{realsolu}), we get
\begin{equation}
	y_{n+1}-\tilde y_{n+1}=(I+Ah)(y_n-\tilde y_n)+\frac{1}{2}y''(w^i,u^i,\rho^i,x^i)h^2-e_{\mathrm{roundoff}}.
\end{equation}
where $A$ is the gradient of the RHS $f$ with respect to $(w,u,\rho)$,
\begin{align}
	\label{amplify}
	A=&\nabla f(x_n,w^4,u^4,\rho^4)\\
	=&\left(\begin{array}{ccc}
\frac{-2c_l-u^4}{c_lx^n+u^4x_n}& \frac{c_l^2(w^4-2x_n^2\rho^4)+4c_l \rho^4x_n^2+c_lu^4w^4}{( u^4+c_l)^3x_n}& \frac{c_l(c_l-2)x_n}{(c_l+u^4)^2}\\	
c_l & 0 &0\\
0 & -\frac{c_l(c_l-2)\rho^4}{(c_l+u^4)^2\eta} &-\frac{c_l(u^4+5)+3u^4}{(c_l+u^4)x^n}\\
\end{array}
	\right)
\end{align}
with $y^4=(w^4,u^4,\rho^4)$ lies within the numerical solution $\tilde y_n$ and real solution $y_n$.

So we get a rigorous error bound at the $n+1$-st step,
\begin{align}
	\label{updateestimate}
	|y_{n+1}-\tilde y_{n+1}|&\leq|I+Ah|E_n+\frac{h^2}{2}|y''(w^i,u^i,\rho^i,x^i)|+|e_{\mathrm{roundoff}}|\\
	&=I_1+I_2+I_3,
\end{align}
where $I_1=|I+Ah|E_n$ is the propagation of error from the previous step, $I_2=h^2/2|y''(y^i,x^i)|$ is the local truncation error, and $I_3$ is the roundoff error in computing $\tilde y_{n+1}$ in (\ref{numericalsolu}).\\

\noindent
{\em Step 4}\quad We first consider the first part $I_1$. We can get lower and upper bounds of $y^4$ using $|y^4-\tilde y_n|\leq E_n$, Lemma~\ref{bmonotone} and Lemma~\ref{boundroundoff}. Since $A$ is explicitly given by (\ref{amplify}), we can get lower and upper bounds for each entry of $I+Ah$ using Lemma~\ref{bmonotone} and Lemma~\ref{boundroundoff}. The strategy is the same as {\em Step 2} and we omit the details here. Taking absolute value for these lower and upper bounds, we get upper bounds for $|I+Ah|$. Then we use Lemma~\ref{bmonotone} and Lemma~\ref{boundroundoff} in computing $|I+Ah|E_n$ and finally get rigorous upper bound of $I_1$.\\

\noindent
{\em Step 5}\quad Then we consider the second part $I_2$. To bound the local truncation error $I_2$, we need to bound the second order derivatives of the solutions on the interval $[x_n,x_{n+1}]$. For $c_l=3$ we have 
\begin{subequations}
	\begin{align}
		\tilde \rho''(\eta)&=\frac{3\tilde \rho(90+71\tilde U+14\tilde U^2-3\eta\tilde W)}{\eta^2(3+\tilde U)^2},\\
		\tilde W''(\eta)&=\frac{-18\eta^2\tilde \rho(\eta)(3+\tilde U+\eta \tilde W)+(3+\tilde U)\tilde W(54+21 \tilde U +2 \tilde U^2+ 9\eta \tilde W)}{\eta^2(3+\tilde U)^3},\\
		\tilde U''(\eta)=&\frac{9\eta^2\tilde \rho-3(18+9\tilde U+\tilde U^2)W}{\eta(3+\tilde U)^2}.
	\end{align}
	\label{2nd}
\end{subequations}

We need the following {\em a priori} estimates to bound (\ref{2nd}).
\begin{Lem}\label{aprioril}
	Consider the ODE system (\ref{solvefar}) with initial conditions given by power series (\ref{taylor}), assume at $\eta_0$, the solutions are $\tilde U(\eta_0)$, $\tilde W(\eta_0)$, $\tilde \rho(\eta_0)$, then for $\eta\in [\eta_0,\eta_0+h]$, we have the following {\em a priori} estimates
	\begin{subequations}
	\begin{equation}
		\tilde \rho(\eta)\in [\rho_{\min},\rho_{\max}],\quad \tilde U(\eta)\in[u_{\min},u_{\max}],\quad \tilde W(\eta)\in [w_{\min},w_{\max}]. 
	\end{equation}
	with 
	\begin{align}
		&\rho_{\max}=\tilde \rho(\eta_0)\eta_0^{c_l-sc_l}(\eta_0+h)^{sc_l-c_l}, &&\rho_{\min}=\tilde \rho(\eta_0)\eta_0^{3+c_l}(\eta_0+h)^{-3-c_l},\\
		 &u_{\min}=\tilde U(\eta_0), &&w_{\max}=\tilde W(\eta_0)+\frac{s^2c_l\rho_{\max}(\eta_0+h)h}{c_l-2},\\
		 &u_{\max}=\tilde U(\eta_0)+w_{\max}h, && w_{\min}=\tilde W(\eta_0)-h\frac{w_{\max}}{\eta_0}(1+\frac{c_l}{c_l+u_{\min}}).
	\end{align}
	\label{apriori}
	\end{subequations}
\end{Lem}
\begin{proof}
	According to (\ref{solvefar1}) and the lower bound of $\tilde U(\eta)$ (\ref{prepare}), we have
	\begin{equation}
		\tilde \rho'(\eta)\leq \frac{\tilde \rho(\eta)}{\eta}(sc_l-c_l),\quad \tilde \rho'(\eta)\geq \frac{\tilde \rho(\eta)}{\eta}(-3-c_l).
	\end{equation}
	By direct integration, we can get $\rho_{\max}$ and $\rho_{\min}$. Since $\tilde U(\eta)$ is increasing according to (\ref{solvefar3}), so we get the lower bound $u_{\min}$. Then using the upper bound $\rho_{\max}$ and the lower bound of $\tilde U(\eta)$ (\ref{prepare}) in (\ref{solvefar2}), we get
\begin{equation}
	\tilde W'(\eta)\leq \frac{s^2c_l\rho_{\max}(\eta_0+h)}{c_l-2}.
\end{equation}
By direct integration we get the upper bound $w_{\max}$. Putting the upper bound of $\tilde W(\eta)$ in (\ref{solvefar3}), we get the upper bound of $\tilde u(\eta)$, $u_{\max}$. Using the upper bound $w_{\max}$ and lower bound $u_{\min}$ in (\ref{solvefar2}), we have
\begin{equation}
	\tilde W'(\eta)\geq\frac{w_{\max}}{\eta_0}(-1-\frac{c_l}{c_l+u_{\min}}),
\end{equation} 
and with this we can get the lower bound of $\tilde W(\eta)$, $w_{\min}$.
\end{proof}

Using $|y_n-\tilde y_n|\leq E_n$ Lemma \ref{bmonotone} and Lemma~\ref{boundroundoff}, we can get lower and upper bounds of $y_n$. Putting these the upper and lower bounds in (\ref{apriori}), and using the same strategy as {\em Step 2}, we can get lower and upper bounds of the solutions on the interval $[x_n,x_{n+1}]$. Finally using these lower and upper bounds of the solutions in (\ref{2nd}) and using the same strategy as in {\em Step 2}, we can rigorously bound $I_2$ using only floating point operation.\\

\noindent
{\em Step 6}\quad The last part we need to control is $I_3$.  $I_3$ is the roundoff error $e_{\mathrm{roundoff}}$ in updating $\tilde y_{n+1}$ using (\ref{numericalsolu}). Using Lemma~\ref{bmonotone} and Lemma~\ref{boundroundoff} and the same strategy as {\em Step 2}, we can get lower and upper bounds of the exact numerical solutions at the $n+1$-th step using only floating point operation. In our computation we use the upper bounds as our numerical solutions at $n+1$-th step, then the roundoff error introduced in this step can be bounded by the difference between the lower and upper bounds. So using Lemma~\ref{bmonotone} and Lemma~\ref{boundroundoff} we can get upper bound of $I_3$ using only floating point operation.\\

\noindent
{\em Step 7}\quad 
Putting the three parts of error bounds together and using Lemma~\ref{bmonotone} and Lemma~\ref{boundroundoff}, we get a rigorous error bound in the next step $E_{n+1}$ using only floating point operation.\\

\noindent
{\em Step 8}\quad Keep updating $\tilde y_n$ and $E_n$ following {\em Steps 3-7}, we finally get the numerical solution of (\ref{solvefar}) at $\eta_0=3$ and a rigorous error bound, which are
\begin{equation}
	 \tilde U(\eta_0)\approx -1.6116,\quad\tilde W(\eta_0)\approx 3.6921\times 10^{-2},\quad \tilde \rho(\eta_0)\approx 3.8422\times 10^{-3}.
\end{equation}
and
\begin{equation}
E_n\leq(5\times10^{-5}, 5\times10^{-4}, 5\times10^{-6})^T.
\end{equation} 

With these numerical solutions and the error bounds, we can easily verify that condition (\ref{assump}) holds, and we complete the proof that for $s=2$,
\begin{equation}
	G(3)<0.
\end{equation}

\subsection{The case $s=2$, $c_l$=8}
The verification of $G(8)>0$ can be done in the same way. In the construction of the local solutions, we can easily verify the choice of
\begin{equation}
	u^0=\frac{1}{6},\quad \rho^0=\frac{1}{18},\quad r=6,
\end{equation}
makes the constraint (\ref{initialbound}) hold. Then we truncate the power series (\ref{taylor}) to the first $20$ terms and evaluate them at $\xi=0.6^8$, corresponding to $\eta_s=0.6$. Using the same technique as the case $c_l=3$, we can get the approximated profiles and error bounds at $\eta_s$,
\begin{equation}
\tilde y_0,\quad E_0=(10^{-13},10^{-13},10^{-13})^T.
\end{equation}

Then we begin to solve (\ref{solvefar}). Using the same technique as the previous case to bound the numerical error in solving the ODE system, we finally get
\begin{equation}
	\tilde U(3)\approx 4.7661\times 10^{-1},
\end{equation}
with numerical error bounded by
\begin{equation}
	E_n^u\leq 2\times 10^{-2}.
\end{equation} 
Then (\ref{assumpp}) holds and we complete the proof that for $s=2$,
\begin{equation}
G(8)>0.
\end{equation}

With $G(3)<0$, $G(8)>0$, we conclude that there exists a $c_l$ such that the self-similar equations have solutions, and the leading order of $\rho(\xi)$ at $\xi=0$ is $s=2$. 

Following the same procedure we also verified that for 
\begin{equation}
s=3,4,5,
\end{equation}
there exists $c_l$ such that the self-similar equations have solutions. We omit the numerical details here for clarity. Now we complete the proof of Theorem~\ref{firstresult}, i.e. there exist a family of self-similar profiles corresponding to different leading orders of $\rho(\xi)$, $s=2,3,4,5$.

\begin{Rem}
	We only verify the existence of self-similar profiles for $s=2,3,4,5$. But the same procedure can be applied to $s>5$ to verify the existence of self-similar profiles.
\end{Rem}

\section{Behavior of the Self-Similar Profiles at Infinity}
\label{farfield}
In this section, we prove that if the decay condition (\ref{gcon}) in the Biot-Savart law is satisfied for the self-similar profiles we construct, then the profiles are actually analytic with respect to a transformed variable $\theta=\xi^{-1/c_l}$ at $\theta=0$. With this we can complete the proof of Theorem~\ref{farana}. This far-field property of the profiles can explain the H\"older continuity of the velocity field at the singularity time observed in numerical simulation of this model.
\begin{Thm}\label{farfieldbehavior}
	For some fixed $c_l$ and $s$, if the self-similar profiles constructed using power series (\ref{taylor}) and extended to whole $R^+$ satisfy the decay condition (\ref{gcon}), then after the following change of variables,  
	\begin{equation}
		\label{farchange}
		\theta=\xi^{-1/c_l},\quad \tilde U(\theta)=U(\xi)\xi^{-1+1/c_l},\quad \tilde \rho(\theta)=\rho(\xi)\xi^{-1+2/c_l},\quad \tilde W(\theta)=W(\xi)\xi^{1/c_l},
	\end{equation}
	$\tilde U(\theta)$, $\tilde W(\theta)$ and $\tilde \rho(\theta)$ are analytic functions at $\theta=0$.
\end{Thm}

Our strategy is the following: We first prove that after the change of variables~(\ref{farchange}), $\tilde U(\theta)$, $\tilde W(\theta)$ and $\tilde \rho(\theta)$ are smooth at $\theta=0$. Then we argue that there exist analytic solutions to the ODE system they satisfy with the same initial conditions at $\theta=0$. Finally we prove that smooth solutions with given initial conditions are unique and complete the proof. 
\begin{proof}
If the decay condition~(\ref{gcon}) holds, $\hat U(\eta)$ tends to $0$ in equation (\ref{farode}), so there exists $\eta_0>0$ such that for $\eta>\eta_0$, 
\begin{equation}
	\frac{(2/c_l-1)\hat U(\eta)}{1+1/c_l\hat U(\eta)}\in (0,1/2).
\end{equation}
Then based on (\ref{farode1}), we have for $\eta>\eta_0$,
\begin{equation}
	\hat\rho'(\eta)\leq \frac{1/2\hat \rho(\eta)}{\eta},
\end{equation}
which implies that for $\eta>\eta_0$,
\begin{equation}
	\label{rhobound}
	\hat \rho(\eta)\leq C_1\eta^{1/2}.
\end{equation}
Using this estimate in (\ref{farode2}), we have for $\eta>\eta_0$,
\begin{equation}
	\big(\hat W(\eta)\eta\big)'\leq C_2\eta^{-3/2},
\end{equation}
which gives
\begin{equation}
	\label{Wbound}
\hat W(\eta)\eta<C_3.
\end{equation}
Using the above estimate in (\ref{farode3}), we get for $\eta>\eta_0$,
\begin{equation}
	\hat U'(\eta)\leq C_4\eta^{-2},
\end{equation}
which together with $\hat U(+\infty)=0$ implies that for $\eta>\eta_0$,
\begin{equation}
	\hat U(\eta)\leq C_5\eta^{-1}.
	\label{sboundu}
\end{equation}
Based on (\ref{farode2}) and (\ref{farode3}), we have
\begin{equation}
	\hat \rho'(\eta)=\frac{(2/c_l-2)\hat \rho(\eta)\hat U(\eta)}{\eta+1/c_l\hat U(\eta)\eta},\quad (\hat W(\eta)\eta)'=\frac{1/c_l\hat U(\eta)\hat W(\eta)}{1+1/c_l\hat U(\eta)}+\frac{(1-2/c_l)\hat \rho(\eta)}{(1+1/c_l\hat U(\eta))^2\eta^2}.	
	\label{farder}
\end{equation}
Using (\ref{sboundu}) and (\ref{rhobound}) in (\ref{farder}), we can see that $|\hat \rho'(\eta)|$  and $|(\hat W(\eta)\eta)'|$ are both integrable from $\eta_0$ to $+\infty$,  thus $\hat \rho(\eta)$ and $\hat W(\eta)\eta$ converge as $\eta\to +\infty$.
\begin{equation}
	\label{limit1}
	\lim_{\eta\to\infty} \hat W(\eta)\eta= \hat W_{\infty}\in[0,+\infty), \quad \lim_{\eta\to\infty} \hat \rho(\eta)= \hat \rho_{\infty}\in(0,+\infty).
\end{equation}
Then based on (\ref{farode3}), and the fact that $\hat U(+\infty)=0$, we have
\begin{equation}
	\label{limit2}
	\lim_{\eta\to+\infty}\hat U(\eta)\eta=c_l\hat W_{\infty}.
\end{equation}

The above limits imply that after changing variables, $\tilde U(\theta)$, $\tilde \rho(\theta)$ and $\tilde W(\theta)$ are continuous for $\theta\in [0,+\infty)$. The ODE system they satisfy for $\theta\in(0,+\infty)$ is
\begin{subequations}
	\begin{align}
		\tilde \rho'(\theta)&=\frac{(2/c_l-1)\tilde\rho(\theta)\tilde U(\theta)}{-1-\tilde U(\theta)\theta},\label{theta1}\\
		\tilde W'(\theta)&=\frac{1/c_l\tilde U(\theta)\tilde W(\theta)+(1-2/c_l)\tilde \rho(\theta)-1/c_l\tilde \rho'(\theta)\theta}{-1-\tilde U(\theta)\theta},\label{theta2}\\
		\tilde U'(\theta)&=-\frac{\tilde U(\theta)}{\theta}-\frac{c_l\tilde W(\theta)}{\theta},\label{theta3}
	\end{align}

	with initial conditions given by (\ref{limit1}) and (\ref{limit2}).
\begin{equation}
	\label{thetaic}
	\tilde W(0)=\hat W_{\infty},\quad \tilde \rho(0)=\hat \rho_\infty,\quad \tilde U(0)=c_l\hat W_\infty.
\end{equation}
	\label{far-field}
\end{subequations}

Equation (\ref{theta3}) can be written as 
\begin{equation}
	\tilde U(\theta)=\frac{c_l}{\theta}\int_0^\theta\tilde W(\zeta)\mathrm{d}\zeta.
\end{equation}
Using a simple bootstrap argument, we can get $\tilde W(\theta)$, $\tilde \rho(\theta)$ and $\tilde U(\theta)$ are in $C^\infty\big([0,+\infty)\big)$. 
On the other hand, given the initial conditions (\ref{thetaic}), we can construct the following power series solutions to equations (\ref{far-field}):
	\begin{equation}
		\label{farpower}
		\tilde U(\theta)=c_l\hat W_\infty+\sum_{k=1}^\infty\tilde U_k\theta^k,\quad 
		\tilde W(\theta)=\hat W_\infty+\sum_{k=1}^\infty\tilde W_k\theta^k,\quad
	 	\tilde \rho(\theta)=\hat \rho_\infty+\sum_{k=1}^\infty\tilde \rho_k\theta^k.
	\end{equation}

	Plugging these power series ansatz in (\ref{far-field}) and matching the coefficients of $\theta^k$, we can uniquely determine the coefficients $\tilde U_k$, $\tilde W_k$, $\tilde \rho_k$ and prove that the power series (\ref{farpower}) converge in a small neighborhood of $\theta=0$. We omit the details here, because the argument are the same as in section~\ref{nearfield}. Then to prove the analyticity of $\tilde{U}(\theta)$, $\tilde{W}(\theta)$ and $\tilde{\rho}(\theta)$ at $\theta=0$, we only need the uniqueness of smooth solutions to (\ref{far-field}) with initial condition (\ref{thetaic}). The RHS of (\ref{theta3}) is not Lipschitz continuous, so the classical result will not apply here.

Assume $\tilde U^i(\theta)$, $\tilde W^i(\theta)$, $\tilde \rho^i(\theta)$, $i=1,2$, are two different solutions to equation (\ref{far-field}) with initial condition (\ref{thetaic}). And let $\delta U(\theta)$, $\delta W(\theta)$, $\delta \rho(\theta)$ be the difference of the two solutions,
\begin{equation}
	\delta \tilde U(\theta)=\tilde U^1(\theta)-\tilde U^2(\theta),\quad
	\delta \tilde W(\theta)=\tilde W^1(\theta)-\tilde W^2(\theta),\quad
	\delta \tilde \rho(\theta)=\tilde \rho^1(\theta)-\tilde \rho^2(\theta).
\end{equation}
Then based on (\ref{theta3}),
	\begin{equation}
		\delta U(\theta)=\frac{c_l}{\theta}\int_0^\theta\delta W(\zeta)\mathrm{d}\zeta.
	\end{equation}
	Using hardy inequality\cite{garling2007}, there exists $C_1$ independent of $\epsilon$ such that
	\begin{equation}
	\label{hardyin}
		\|\delta \tilde U\|_{L^2([0,\epsilon])}\leq C_1\|\delta\tilde W\|_{L^2([0,\epsilon])}.
	\end{equation}
	Since the right hand side of (\ref{theta1}) and (\ref{theta2}) are Lipschitz continuous, we have
	\begin{equation}
		|\frac{\mathrm{d}}{\mathrm{d}\theta}(\delta\tilde W(\theta))|+|\frac{\mathrm{d}}{\mathrm{d}\theta}(\delta\tilde \rho(\theta)) )|\leq C_2(|\delta\tilde W(\theta)|+|\delta\tilde U(\theta)|+|\delta\tilde \rho(\theta)|)
	\end{equation}
	Integrating the square of both sides on the interval $[0,\epsilon]$ and using (\ref{hardyin}), we get
	\begin{equation}
		\label{lipin}
		\|\big(\delta \tilde W(\theta)\big)'\|_{L^2([0,\epsilon])}+\|\big(\delta \tilde \rho(\theta)\big)'\|_{L^2([0,\epsilon])}\leq C_3(\|\delta \tilde W(\theta)\|_{L^2([0,\epsilon])}+\|\delta \tilde \rho(\theta)\|_{L^2([0,\epsilon])}).
	\end{equation}
Since $\delta\tilde W(\theta)$ and $\delta\tilde \rho(\theta)$ vanish on $\theta=0$, by Poincar\'e-Friedrichs inequality we have
\begin{equation}
	\label{pfin}
	\|\delta \tilde W(\theta)\|_{L^2([0,\epsilon])}+\|\delta \tilde \rho(\theta)\|_{L^2([0,\epsilon])}\leq C_4\epsilon(\|\big(\delta \tilde W(\theta)\big)'\|_{L^2([0,\epsilon])}+\|\big(\delta \tilde \rho(\theta)\big)'\|_{L^2([0,\epsilon])}). 
\end{equation}

The $C$ in the above estimates are all positive constants independent of $\epsilon$. Choosing $\epsilon$ small enough, we get a contradiction in (\ref{lipin}) and (\ref{pfin}), thus
\begin{equation}
	\tilde W^1=\tilde W^2, \quad \tilde U^1=\tilde U^2,\quad \tilde \rho^1=\tilde \rho^2,
\end{equation}
which means the solution is unique. And we complete the proof of this Theorem.
\end{proof}

The above Theorem implies the self-similar profiles we construct are non-conventional in the sense that the velocity does not decay to $0$ at $+\infty$ but grows with certain fractional power. Coming back to the self-similar ansatz (\ref{ansatz}), we have
\begin{equation}
	u(x,t)=(T-t)^{c_l-1}U\left(\frac{x}{(T-t)^{c_l}}\right).
\end{equation}
For $t$ close to $T$, based on Theorem~\ref{farfieldbehavior} , we have
\begin{equation}
	\label{farholder}
	u(x,t)\approx C(T-t)^{c_l-1}\frac{x}{(T-t)^{c_l}}\left(\frac{x}{(T-t)^{c_l}}\right)^{-\frac{1}{c_l}}=Cx^{1-\frac{1}{c_l}}.
\end{equation}

This can explain the H\"older continuity of the velocity field near singularity time, which is observed in simulation of the 1D model. Such behavior is also observed in the numerical computation of the 3D Euler equations in ~\cite{hou2013finite}, which is very different from the Leray type of self-similar solutions of the 3D Euler equations, whose existence has been ruled out under certain decay assumptions on the self-similar profiles \cite{chae2007nonexistence, chae2007nonexistence2, chae2011self}. 

\section{Numerical Results}
\label{compare}
In this section we numerically locate $c_l$ which makes $G(c_l)=0$ for several different $s$ and construct the corresponding self-similar profiles. The scaling exponents and self-similar profiles obtained from solving the self-similar equations (\ref{targ}) agree with those obtained from direct numerical simulation of the CKY model. We also find that for fixed $s$, the singular solutions using different initial conditions converge to the same self-similar profiles after rescaling, which means the self-similar profiles we find have some stability property.

\subsection{Numerical methods for solving the self-similar equations}
For any fixed $c_l>2$, we first numerically compute the coefficients $\rho_k$, $U_k$ in (\ref{taylor}) up to $k=50$ and determine the convergence radius of the power series using the following linear regression for $s\leq k\leq 50$,
\begin{equation}
	\log \rho_k=k \log r_1+c_1,\quad \log U_k=k\log r_2+c_2.
\end{equation}
We choose
\begin{equation}
	r=\frac{1}{2}\min\{\frac{1}{r_1},\frac{1}{r_2}\}.
\end{equation}
and construct the local self-similar profiles on $[0,r/2]$.

Then we continue solving equation~(\ref{targ}) from $\xi=r/2$ to $\xi=1$ using the 4th order explicit Runge-Kutta method with step-size $h=\frac{1-r/2}{10^4}$. After $\xi=1$, we make the change of variables (\ref{changeeta}) and begin to solve (\ref{farode}) from $\eta=1$ to $\eta=10^5$ using 4th order Runge-Kutta method with step-size $h=\frac{10^5-1}{10^6}$. We use $\hat U_{c_l}(10^5)$ as an approximation to $G(c_l)$.

We use the bisection method to find the root of $G(c_l)$, and the stopping criterion is that the length of the subinterval becomes less than $10^{-5}$. After we get the scaling exponent $c_l$, we construct the local self-similar profiles using power series (\ref{taylor}) as before and numerically extend them from $\xi=r/2$ to $\xi=10$ using the explicit 4th order Runge-Kutta method with step-size $h=\frac{9}{10^4}$. Then we locate the maxima of $W$, which is $W_{\max}=W(\xi_0)$. We consider $s=2, 3, 4, 5$, and for these cases $\xi_0$ are all less than $10$. Finally we rescale the maximum of the $W(\xi)$ to $(1,1)$, and  get the rescaled self-similar profile of $w$
\begin{equation}
	\label{selfsimilarprofile}
	W_s(\xi)=\frac{1}{W_{\max}}W(\xi\xi_0),\quad \xi\in [0,1].
\end{equation}

We will only compare the self-similar profiles of $W_s$ with direct simulation of the CKY model in this paper, but the numerical results for the profiles of $\rho$ and $U$ are similar.

\subsection{Numerical methods for simulating the model}
In the direct simulation of the CKY model, we use a particle method. We consider $N+1$ particles with position, density and vorticity given by
\begin{equation}
	\begin{cases} 
		q=(q_0(t),q_1(t),\dots q_N(t))^T,\\
		\rho=(\rho_0(t),\rho_1(t),\dots \rho_N(t))^T,\\
		w=(w_0(t),w_1(t),\dots w_N(t))^T.
	\end{cases}
\end{equation}
In computing the velocity field, we use the trapezoidal rule to approximate (\ref{sBS}),
\begin{equation}
	u_i=-q_i\left(\sum_{j=i}^{N-1}\frac{w_j+w_{j+1}}{2}(q_{j+1}-q_{j})\right).
\end{equation}
In computing the driving force of $w$, which is $\rho_x$, we use the three points rule:
\begin{equation}
	(\rho_x)_i=
	\begin{cases}
	0,\quad &i=0,\\
	\frac{\rho_i-\rho_{i+1}}{q_i-q_{i+1}}+\frac{\rho_i-\rho_{i-1}}{q_i-q_{i-1}}+\frac{\rho_{i+1}-\rho_{i-1}}{q_{i+1}-q_{i-1}},\quad &0<i<N,\\
	\frac{\rho_i-\rho_{i-2}}{q_i-q_{i-2}}+\frac{\rho_i-\rho_{i-1}}{q_i-q_{i-1}}+\frac{\rho_{i-2}-\rho_{i-1}}{q_{i-2}-q_{i-1}}, &i=N.
\end{cases}
\end{equation}
Initially, $10^5+1$ particles are equally put in the short interval $[0,10^{-3}]$, which are sufficient to resolve the solutions in the self-similar regime. Outside this short interval $10^5-10^2$ particles are equally placed with distance $10^{-5}$. So the total number of particles is $N=2\times 10^5-10^2$.

Then we need to solve the following ODE system
\begin{equation}
	\label{directcky}
		\frac{\mathrm{d}}{\mathrm{d}t}q=u,\quad \frac{\mathrm{d}}{\mathrm{d}t}w=\rho_x,\quad \frac{\mathrm{d}}{\mathrm{d}t}\rho=0.
\end{equation}
The initial condition of $\rho$ is
\begin{equation}
	\rho(x,0)=(1-\cos(\pi x))^{s/2},
\end{equation}
whose leading order at $x=0$ is $s$.

We solve the ODE system~\ref{directcky}  using the $4$-th order explicit Runge-Kutta method, and the time step $dt$ is chosen adaptively to avoid the particles cross each other:
\begin{equation}
	dt_i=\frac{1}{\max(\frac{u_i-u_{i+1}}{q_{i+1}-q_i},0)}, \quad dt=\min(\frac{dt_i}{10},10^{-3}).
\end{equation}

Simulation of this model will stop once the maximal vorticity reaches some preset $W_{\max}$. And at each time step, we record the maximal vorticity $w_{\max}(t_i)$, and the position where it is attained $q_{\max}(t_i)$. According to the self-similar ansatz, we have
\begin{equation}
	w_{\max}(t)=C_1(T-t)^{c_w},\quad q_{\max}(t)=C_2(T-t)^{c_l}.
\end{equation}
Thus we can compute $c_l$ $c_w$, and the singularity time $T$ by doing linear regression,
\begin{subequations}
	\begin{align}
		(\frac{\mathrm{d}}{\mathrm{d}t}\log w_{\max}(t))^{-1}&\approx-\frac{1}{c_w}t+\frac{T}{c_w},\label{wregress}\\
		(\frac{\mathrm{d}}{\mathrm{d}t}\log q_{\max}(t))^{-1}&\approx-\frac{1}{c_l}t+\frac{T}{c_l}.\label{qregress}
	\end{align}
\end{subequations}

We compute the time derivatives of $\log w_{\max}(t)$ and $\log q_{\max}(t)$ using the center difference method, and the linear regressions are done in some time interval close to the singularity time while the numerical solutions still have good accuracy.

At certain time steps close to the singularity time, $t^i$, $i=1,2,3$, let $w^i$ be the maximal vorticity at time $t^i$ and $q^i$ be the position the maximal vorticity is attained. We rescale the numerical solution and get the self-similar profile of $w$,
\begin{equation}
	\label{profiledirect}
	W^i_s(\xi)=\frac{1}{w_{\max}}w(\xi q^i, t^i),\quad \xi\in[0,1].
\end{equation}

In the next subsection we will compare the self-similar profiles $W^i_s(\xi)$ (\ref{profiledirect}), which are obtained from direct simulation of the model, with $W_s(\xi)$ given by  (\ref{selfsimilarprofile}), which is obtained from solving the self-similar equations (\ref{targ}). 

Near the singularity time the velocity field seems to be H\"older continuous near the origin, 
\begin{equation}
	u(x,T)\approx Cx^\alpha.
\end{equation}
Then we can determine the H\"older exponent $\alpha$ by doing linear regression
\begin{equation}
	\label{aregress}
	\ln u(x,T)\approx\ln C+\alpha\ln x.
\end{equation}
We will compare the exponents $\alpha$ we get from direct simulation of the CKY model with our prediction (\ref{farholder})
obtained from analyzing the self-similar equations. 
\subsection{Comparison results}
In our direct simulation of the CKY model, we first choose the initial condition of $w(x,t)$ as
\begin{equation}
	\label{1stw}
	w(x,0)=1-\cos(4\pi x).
\end{equation}
We compute the scaling exponents $c_w$ and $c_l$ for different leading orders of $\rho$, $s=2,3,4,5$ using direct simulation of the CKY model, and the results are listed in Table~\ref{tab:cwtable} and Table~\ref{tab:cltable}. We also compute the H\"older exponents of the velocity field near the origin at the singularity time and compare them with $1-1/c_l$ as predicted by (\ref{farholder}). The results are listed in Table~\ref{tab:holder}. The $c_l$ we use are those obtained from solving the self-similar equations.

For $s=2$, the linear regression (\ref{wregress}) and (\ref{qregress}) are done in the time interval 
\begin{equation}
	[6.4371\times 10^{-1}, 6.4391\times 10^{-1}],
\end{equation}
and the predicted singularity time $T$ for (\ref{wregress}) and (\ref{qregress}) are both $6.4402\times 10^{-1}$. The linear regression (\ref{aregress}) is done at $t=6.4391\times 10^{-1}$ and on the interval $[10^{-10},10^{-9}]$. 

For $s=3$, the linear regression (\ref{wregress}) and (\ref{qregress}) are done in the time interval 
\begin{equation}
	[6.804297\times 10^{-1}, 6.804300\times 10^{-1}],
\end{equation}
and the predicted singularity time $T$ for (\ref{wregress}) and (\ref{qregress}) are both $6.804302\times 10^{-1}$. The linear regression (\ref{aregress}) is done at $t=6.804302\times 10^{-1}$ and on the interval $[10^{-10},10^{-9}]$. 

For $s=4$, the linear regression (\ref{wregress}) and (\ref{qregress}) are done in the time interval 
\begin{equation}
	[6.571218\times 10^{-1}, 6.571221\times 10^{-1}],
\end{equation}
and the predicted singularity time $T$ for (\ref{wregress}) and (\ref{qregress}) are both $6.571223\times 10^{-1}$. The linear regression (\ref{aregress}) is done at $t=6.571223\times 10^{-1}$ and on the interval $[10^{-10},10^{-9}]$. 

For $s=5$, the linear regression (\ref{wregress}) and (\ref{qregress}) are done in the time interval 
\begin{equation}
	[5.9698511\times 10^{-1}, 5.9698515\times 10^{-1}],
\end{equation}
and the predicted singularity time $T$ for (\ref{wregress}) and (\ref{qregress}) are both $5.9698517\times 10^{-1}$. The linear regression (\ref{aregress}) is done at $t=5.9698517\times 10^{-1}$ and on the interval $[10^{-10},10^{-9}]$. 
\begin{table}
	\centering
	\begin{tabular}{|c|c|c|c|c|}
		\hline
		& $s=2$ & $s=3$ & $s=4$ & $s=5$\\
		\hline
		$c_w$& $-0.9747$ & $-1.0001$& $-1.0006$ & $-1.0007$ \\
		\hline
	\end{tabular}
	\caption{$c_w$ table}
	\label{tab:cwtable}
\end{table}

From the Table~\ref{tab:cwtable}, \ref{tab:cltable} \ref{tab:holder}, we can see that the exponents $c_w$ we obtain from the singular numerical solutions are close to $-1$ (\ref{rexponent}). The $c_l$ we obtain from direct simulation of this model are close to those obtained from solving the self-similar equations. And at the singularity time the H\"older exponents of the velocity field are close to $1-1/c_l$. 

For the case $s=2$, the dependence of $G(c_l)$ on $c_l$ is plotted in Figure (\ref{fig:gcl}). From this Figure, we can see that $G(c_l)$ depends continuously on $c_l$ as we have proved. Besides, $G(c_l)$ seems to be a monotone increasing function, which implies that for fixed $s$, the scaling exponent $c_l$ to make the decay condition~(\ref{gcon}) hold is unique.
\begin{figure}[htpb].
	\begin{center}
		\includegraphics[width=0.5\textwidth]{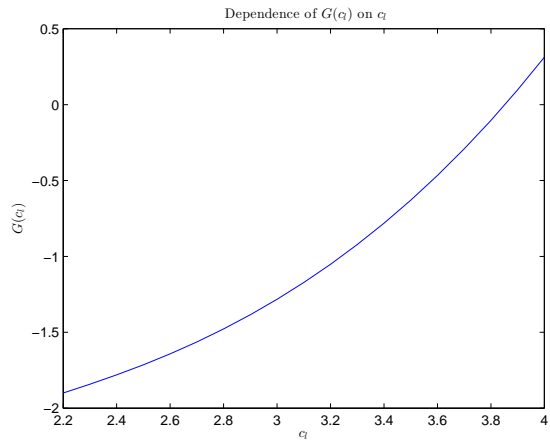}
	\end{center}
	\caption{Dependence of $G(c_l)$ on $c_l$ for $s=2$}
	\label{fig:gcl}
\end{figure}

\begin{table}
	\centering
	\begin{tabular}{|c|c|c|c|c|}
		\hline
		& $s=2$ & $s=3$ & $s=4$ & $s=5$\\
		\hline
		CKY& $3.7942$ & $3.3143$& $3.1718$ & $3.0773$ \\
		\hline
		Self-Similar& $3.7967$ & $3.3157$& $3.1597$& $3.0841$ \\
		\hline
	\end{tabular}
	\caption{$c_l$ table}
	\label{tab:cltable}
\end{table}
\begin{table}
	\centering
	\begin{tabular}{|c|c|c|c|c|}
		\hline
		& $s=2$ & $s=3$ & $s=4$ & $s=5$\\
		\hline
		H\"older exponent& $7.3381\times 10^{-1}$ & $6.9823\times 10^{-1}$& $6.9131\times 10^{-1}$ & $6.7610\times 10^{-1}$ \\
		\hline
		$1-1/c_l$& $7.3661\times 10^{-1}$ & $6.9841\times 10^{-1}$& $6.8351\times 10^{-1}$& $6.7576\times 10^{-1}$ \\
		\hline
	\end{tabular}
	\caption{H\"older exponent of the velocity field at $0$ and $1-1/c_l$.}
	\label{tab:holder}
\end{table}

We also compare the self-similar profiles obtained from solving (\ref{targ}) and those obtained from direct simulation of the model. For different $s$, they are plotted in Figure~\ref{fig:wfile}. The lines labeled `exact' are obtained from solving the self-similar equation. Others are obtained from rescaling the solution at different time steps corresponding to different maximal vorticity.
\begin{figure}[htpb]
\centering
\begin{subfigure}{.5\textwidth}
\centering
\includegraphics[width=0.7\textwidth]{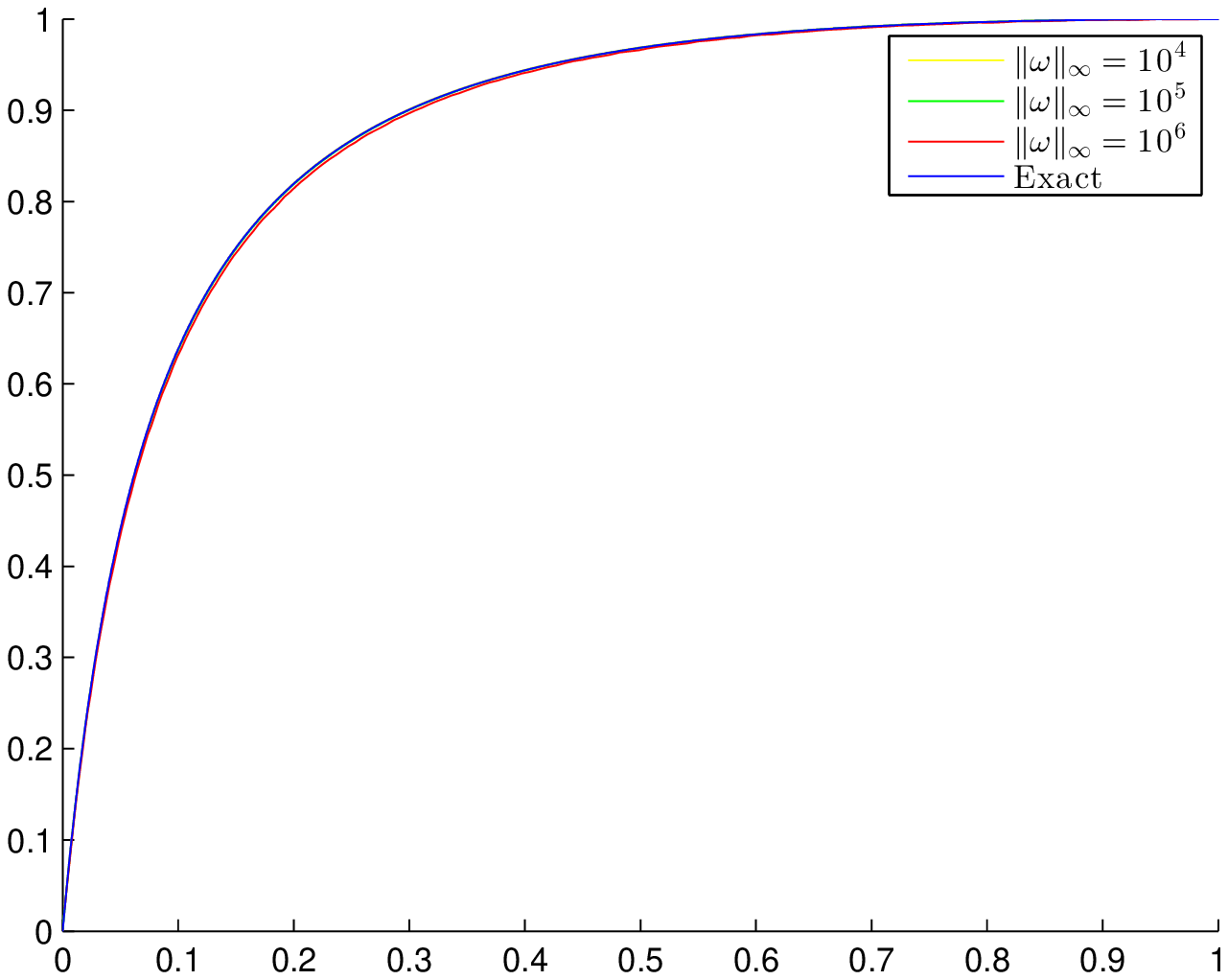}
\caption{The re-scaled solutions and self-similar profiles we construct. $s=2$.}
\end{subfigure}%
\begin{subfigure}{.5\textwidth}
\centering
\includegraphics[width=0.7\textwidth]{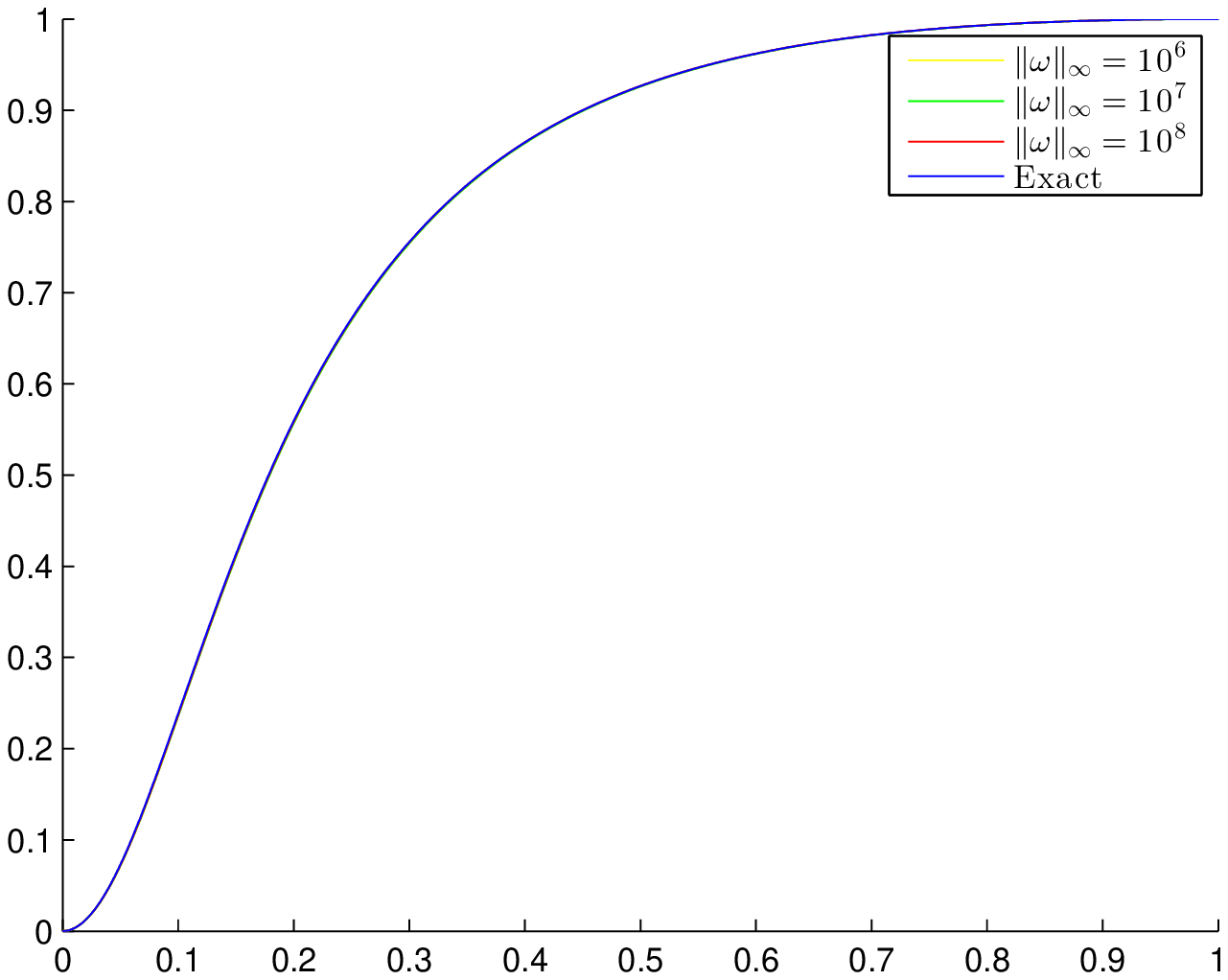}
\caption{The re-scaled solutions and self-similar profiles we construct. $s=3$.}
\end{subfigure}
\begin{subfigure}{.5\textwidth}
\centering
\includegraphics[width=0.7\textwidth]{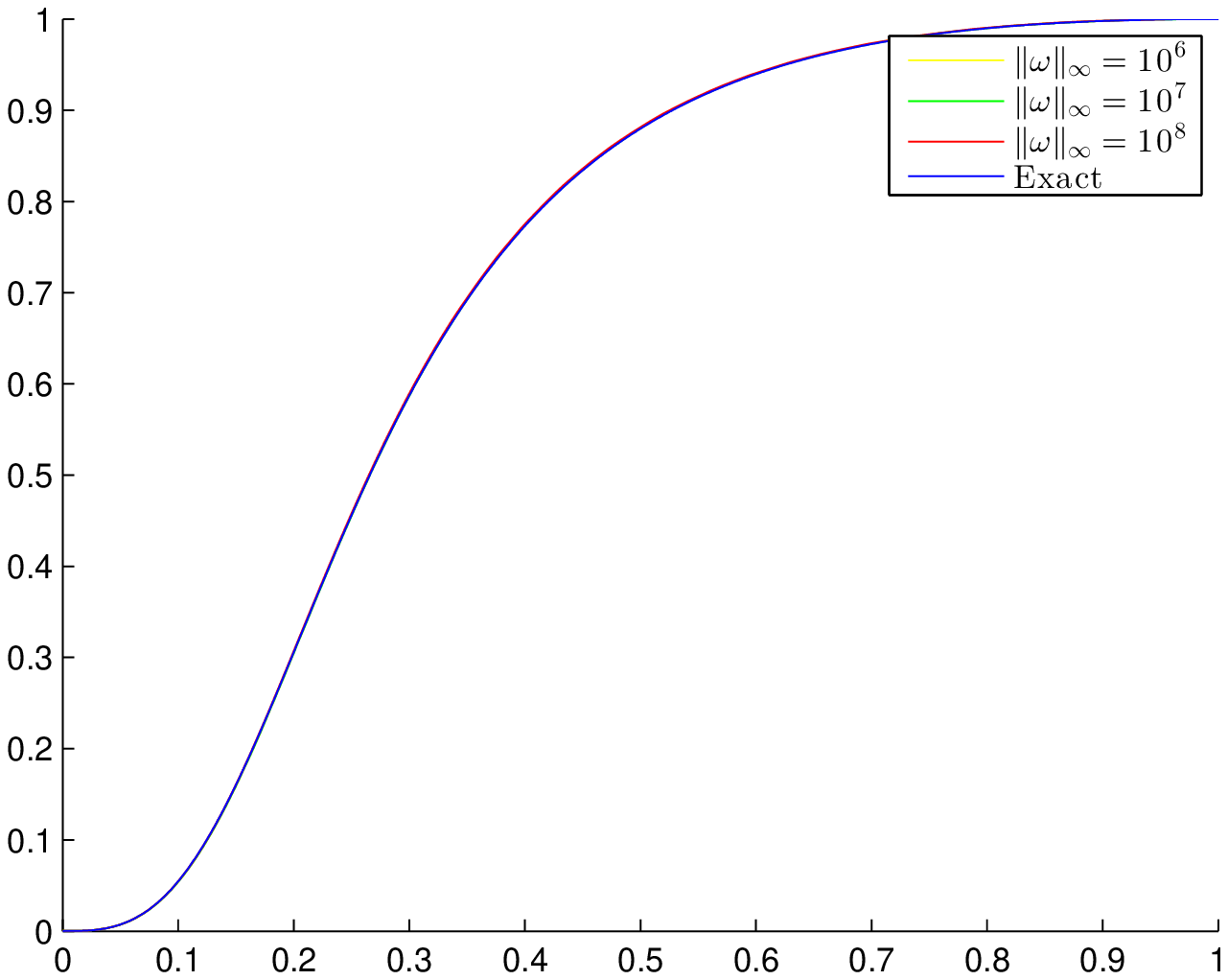}
\caption{The re-scaled solutions and self-similar profiles we construct. $s=4$.}
\end{subfigure}%
\begin{subfigure}{.5\textwidth}
\centering
\includegraphics[width=0.7\textwidth]{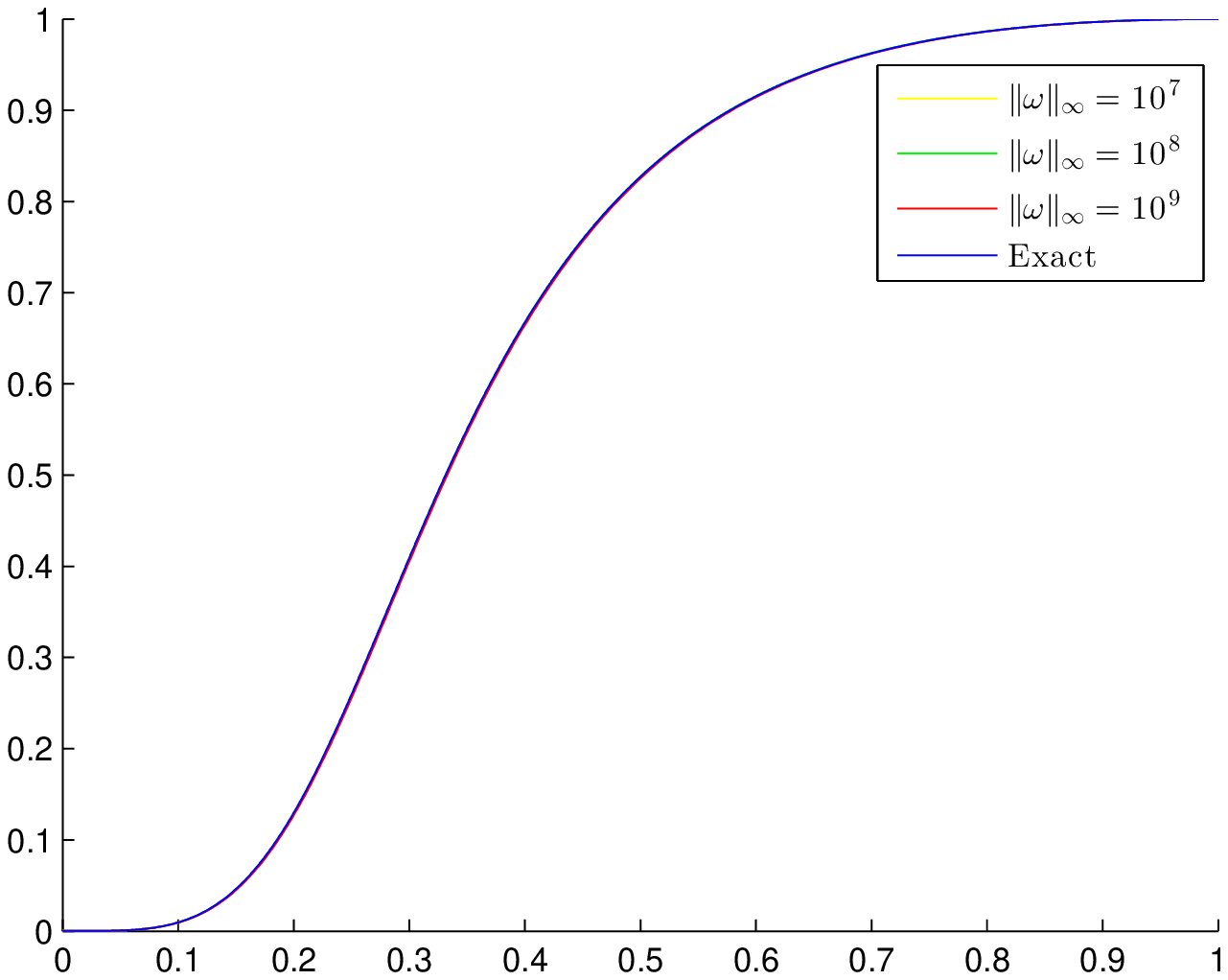}
\caption{The re-scaled solutions and self-similar profiles we construct. $s=5$.}
\end{subfigure}
\caption{Self-similar profiles of $w$ using initial condition $w(x,0)=1-\cos(4\pi x)$. }\label{fig:wfile}
\end{figure}

To demonstrate the stability the self-similar profiles, we consider another initial condition,
\begin{equation}
	w(x,0)=x-x^2.
	\label{2ndw}
\end{equation}
Again we compare the self-similar profiles of $w$ from direct simulation of the model with those obtained from solving the self-similar equations (\ref{targ}). They are plotted in Figure~\ref{fig:wfile2}.
\begin{figure}[htpb]
\centering
\begin{subfigure}{.5\textwidth}
\centering
\includegraphics[width=0.7\textwidth]{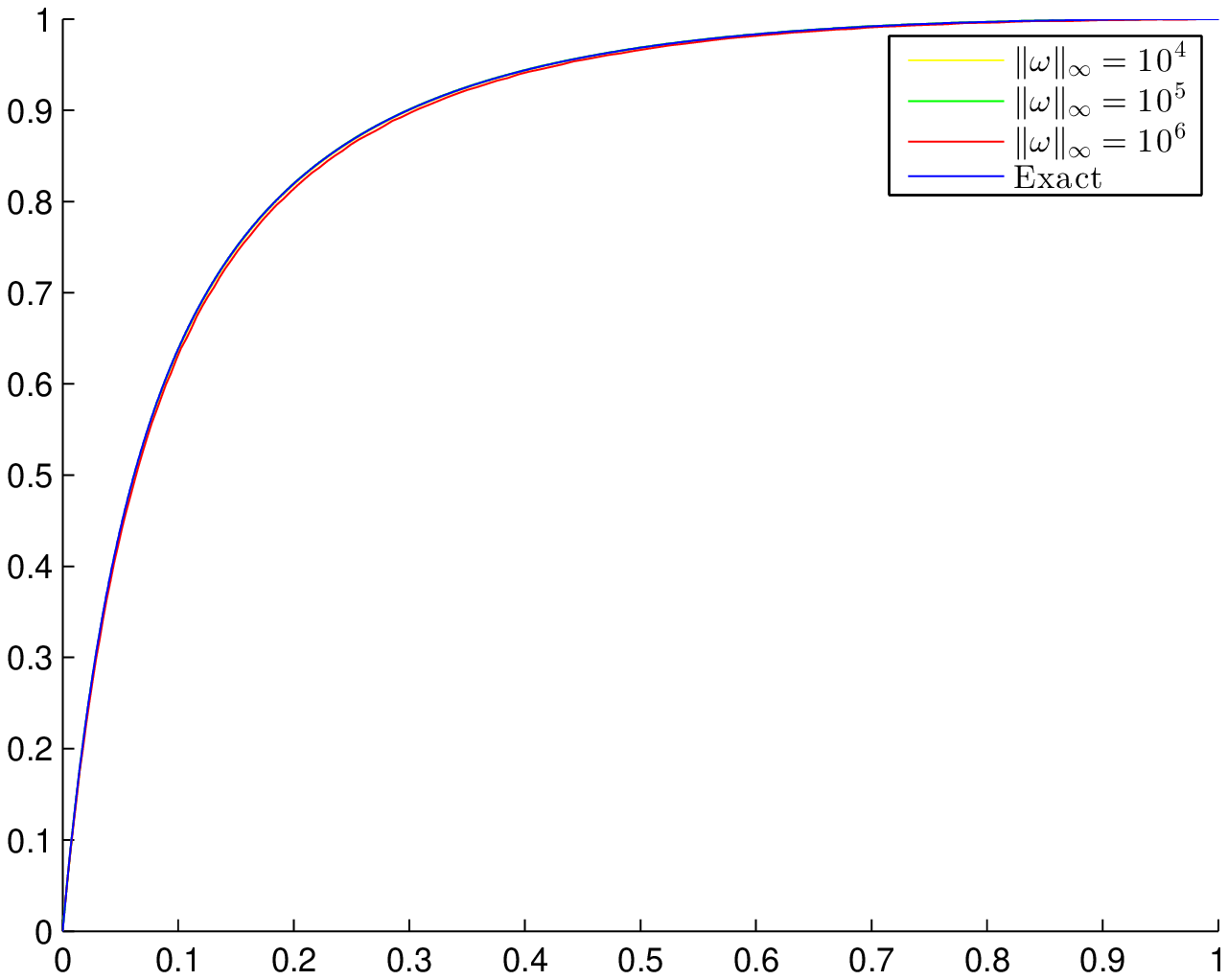}
\caption{$s=2$}
\end{subfigure}%
\begin{subfigure}{.5\textwidth}
\centering
\includegraphics[width=0.7\textwidth]{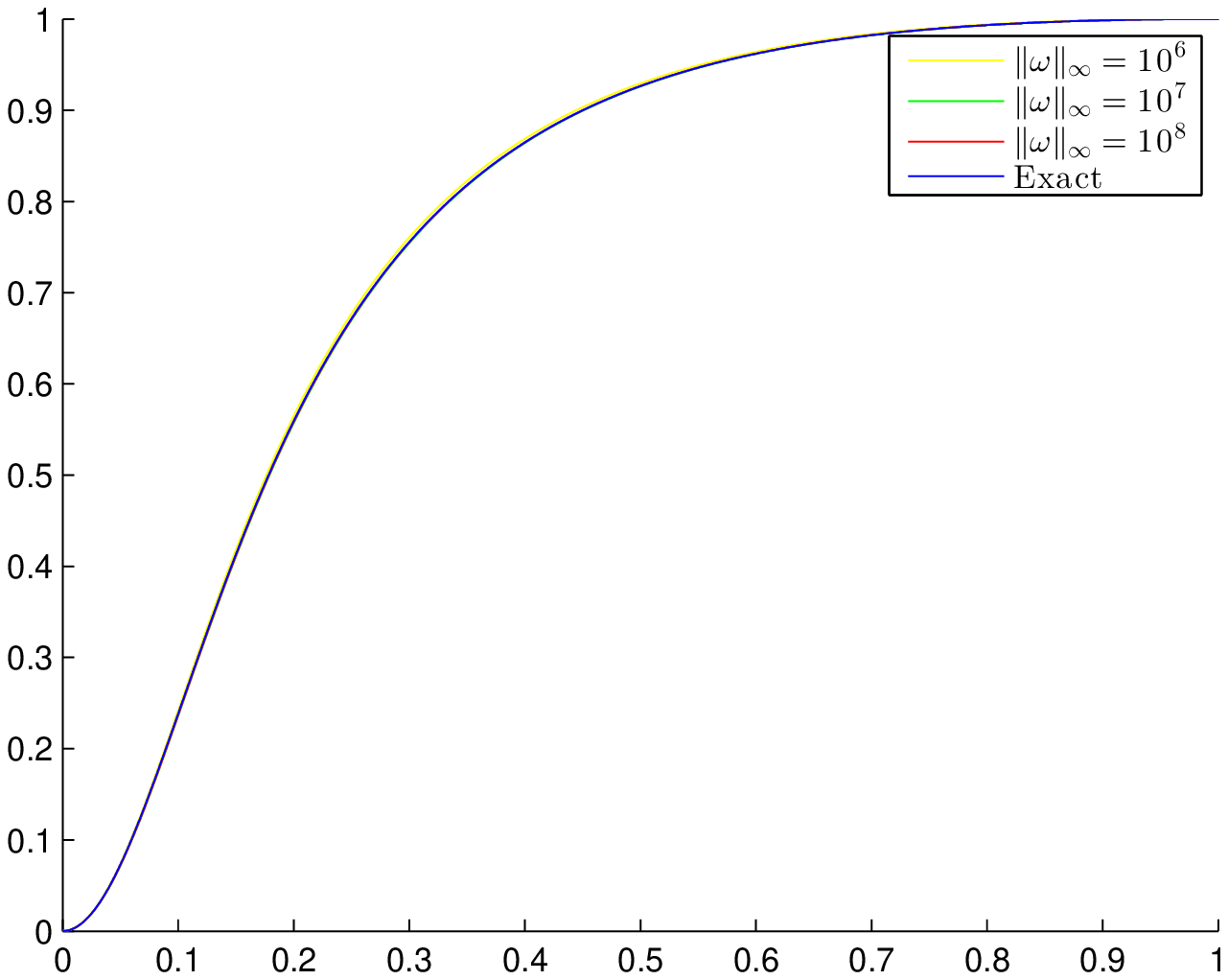}
\caption{$s=3$}
\end{subfigure}
\caption{Self-similar profiles of $w$ using initial condition $w(x,0)=x-x^2$.}\label{fig:wfile2}
\end{figure}

From Figure~\ref{fig:wfile},~\ref{fig:wfile2}, we can see that after rescaling, the singular solutions at different time steps before the singularity time are very close, which implies that the solutions of this 1D model develop self-similar singularity. And the self-similar profiles obtained from direct simulation of the model (\ref{profiledirect}) agree very well with the self-similar profiles we construct (\ref{selfsimilarprofile}) by sovling the self-similar equations (\ref{targ}). For fixed leading order of $\rho(x,0)$ at the origin, the singular solutions with different initial conditions converge to the same set of self-similar profiles, which implies that the profiles we construct have some stability property.

\vspace{0.2in}
{\bf Acknowledgements.} The authors would like to thank Professors Russel Caflisch and Guo Luo for a number of stimulating discussions. We would also like to thank Professors Alexander Kiselev and Yao Yao for their interest in our work and for their valuable comments. The research was in part supported by NSF FRG Grant DMS-1159138.

\bibliographystyle{plain}
\bibliography{selfsimilar}

\begin{thebibliography}{10}

\bibitem{appel1976proof}
Kenneth Appel and Wolfgang Haken.
\newblock Proof of 4-color theorem.
\newblock {\em Discrete Mathematics}, 16(2):179--180, 1976.

\bibitem{bardos2007euler}
Claude Bardos and Edriss Titi.
\newblock Euler equations for incompressible ideal fluids.
\newblock {\em Russian Mathematical Surveys}, 62(3):409, 2007.

\bibitem{chae2007nonexistence2}
Dongho Chae.
\newblock Nonexistence of asymptotically self-similar singularities in the
  euler and the navier--stokes equations.
\newblock {\em Mathematische Annalen}, 338(2):435--449, 2007.

\bibitem{chae2007nonexistence}
Dongho Chae.
\newblock Nonexistence of self-similar singularities for the 3d incompressible
  euler equations.
\newblock {\em Communications in Mathematical Physics}, 273(1):203--215, 2007.

\bibitem{chae2011self}
Dongho Chae.
\newblock On the self-similar solutions of the 3d euler and the related
  equations.
\newblock {\em Communications in Mathematical Physics}, 305(2):333--349, 2011.

\bibitem{choi2014on}
Kyudong Choi, Thomas~Y. Hou, Alexander Kiselev, Guo Luo, Vladimir Sverak, and
  Yao Yao.
\newblock On the fiinite-time blowup of a 1d model for the 3d axisymmetric
  euler equations.
\newblock {\em arXiv preprint arXiv:1407.4776}, 2014.

\bibitem{choi2013finite}
Kyudong Choi, Alexander Kiselev, and Yao Yao.
\newblock Finite time blow up for a 1d model of 2d boussinesq system.
\newblock {\em arXiv preprint arXiv:1312.4913}, 2013.

\bibitem{coddington1955theory}
Earl~A Coddington and Norman Levinson.
\newblock {\em Theory of ordinary differential equations}.
\newblock Tata McGraw-Hill Education, 1955.

\bibitem{constantin2007euler}
Peter Constantin.
\newblock On the euler equations of incompressible fluids.
\newblock {\em Bulletin of the American Mathematical Society}, 44(4):603--621,
  2007.

\bibitem{fefferman1996interval}
Charles~L Fefferman and Luis~A Seco.
\newblock Interval arithmetic in quantum mechanics.
\newblock In {\em Applications of interval computations}, pages 145--167.
  Springer, 1996.

\bibitem{folland1995}
Gerald~B Folland.
\newblock {\em Introduction to partial differential equations}.
\newblock Princeton University Press, 1995.

\bibitem{garling2007}
D.J.H. Garling.
\newblock {\em Inequalities: a journey into linear analysis}, volume~19.
\newblock Cambridge University Press Cambridge, 2007.

\bibitem{gibbon2008three}
John~D Gibbon.
\newblock The three-dimensional euler equations: Where do we stand?
\newblock {\em Physica D: Nonlinear Phenomena}, 237(14):1894--1904, 2008.

\bibitem{hales2005proof}
Thomas~C Hales.
\newblock A proof of the kepler conjecture.
\newblock {\em Annals of mathematics}, pages 1065--1185, 2005.

\bibitem{hou2013finite}
Thomas~Y Hou and Guo Luo.
\newblock On the finite-time blowup of a 1d model for the 3d incompressible
  euler equations.
\newblock {\em arXiv preprint arXiv:1311.2613}, 2013.

\bibitem{kearfott1996applications}
R~Baker Kearfott and Vladik Kreinovich.
\newblock {\em Applications of interval computations}, volume~3.
\newblock Kluwer Academic Dordrecht, 1996.

\bibitem{kiselev2013small}
Alexander Kiselev and Vladimir Sverak.
\newblock Small scale creation for solutions of the incompressible two
  dimensional euler equation.
\newblock {\em arXiv preprint arXiv:1310.4799}, 2013.

\bibitem{kovalevskaja1874theorie}
Sof'ja~V Kovalevskaja.
\newblock Zur theorie der partiellen differentialgleichungen.
\newblock 1874.

\bibitem{lanford1982computer}
Oscar~E Lanford~III.
\newblock A computer-assisted proof of the feigenbaum conjectures.
\newblock {\em Bulletin of the American Mathematical Society}, 6(3):427--434,
  1982.

\bibitem{leveque2007finite}
Randall~J LeVeque.
\newblock {\em Finite difference methods for ordinary and partial differential
  equations: steady-state and time-dependent problems}, volume~98.
\newblock Siam, 2007.

\bibitem{luo2013potentially}
Guo Luo and Thomas~Y Hou.
\newblock Potentially singular solutions of the 3d incompressible euler
  equations.
\newblock {\em arXiv preprint arXiv:1310.0497}, 2013.

\bibitem{majda2002vorticity}
Andrew~J Majda and Andrea~L Bertozzi.
\newblock {\em Vorticity and incompressible flow}, volume~27.
\newblock Cambridge University Press, 2002.

\bibitem{2008ieee}
Dan Zuras, Mike Cowlishaw, Alex Aiken, Matthew Applegate, David Bailey, Steve
  Bass, Dileep Bhandarkar, Mahesh Bhat, David Bindel, Sylvie Boldo, et~al.
\newblock Ieee standard for floating-point arithmetic.
\newblock {\em IEEE Std 754-2008}, pages 1--70, 2008.

\end{thebibliography}


\begin{thebibliography}{10}

\bibitem{chae2007nonexistence2}
Dongho Chae.
\newblock Nonexistence of asymptotically self-similar singularities in the
  euler and the navier--stokes equations.
\newblock {\em Mathematische Annalen}, 338(2):435--449, 2007.

\bibitem{chae2007nonexistence}
Dongho Chae.
\newblock Nonexistence of self-similar singularities for the 3d incompressible
  euler equations.
\newblock {\em Communications in Mathematical Physics}, 273(1):203--215, 2007.

\bibitem{chae2011self}
Dongho Chae.
\newblock On the self-similar solutions of the 3d euler and the related
  equations.
\newblock {\em Communications in Mathematical Physics}, 305(2):333--349, 2011.

\bibitem{choi2014on}
Kyudong Choi, Thomas~Y. Hou, Alexander Kiselev, Guo Luo, Vladimir Sverak, and
  Yao Yao.
\newblock On the fiinite-time blowup of a 1d model for the 3d axisymmetric
  euler equations.
\newblock {\em arXiv preprint arXiv:1407.4776}, 2014.

\bibitem{choi2013finite}
Kyudong Choi, Alexander Kiselev, and Yao Yao.
\newblock Finite time blow up for a 1d model of 2d boussinesq system.
\newblock {\em arXiv preprint arXiv:1312.4913}, 2013.

\bibitem{garling2007}
D.J.H. Garling.
\newblock {\em Inequalities: a journey into linear analysis}, volume~19.
\newblock Cambridge University Press Cambridge, 2007.

\bibitem{gibbon2008three}
John~D Gibbon.
\newblock The three-dimensional euler equations: Where do we stand?
\newblock {\em Physica D: Nonlinear Phenomena}, 237(14):1894--1904, 2008.

\bibitem{hou2013finite}
Thomas~Y Hou and Guo Luo.
\newblock On the finite-time blowup of a 1d model for the 3d incompressible
  euler equations.
\newblock {\em arXiv preprint arXiv:1311.2613}, 2013.

\bibitem{kiselev2013small}
Alexander Kiselev and Vladimir Sverak.
\newblock Small scale creation for solutions of the incompressible two
  dimensional euler equation.
\newblock {\em arXiv preprint arXiv:1310.4799}, 2013.

\bibitem{luo2013potentially}
Guo Luo and Thomas~Y Hou.
\newblock Potentially singular solutions of the 3d incompressible euler
  equations.
\newblock {\em arXiv preprint arXiv:1310.0497}, 2013.

\bibitem{majda2002vorticity}
Andrew~J Majda and Andrea~L Bertozzi.
\newblock {\em Vorticity and incompressible flow}, volume~27.
\newblock Cambridge University Press, 2002.

\bibitem{2008ieee}
Dan Zuras, Mike Cowlishaw, Alex Aiken, Matthew Applegate, David Bailey, Steve
  Bass, Dileep Bhandarkar, Mahesh Bhat, David Bindel, Sylvie Boldo, et~al.
\newblock Ieee standard for floating-point arithmetic.
\newblock {\em IEEE Std 754-2008}, pages 1--70, 2008.

\end{thebibliography}

\end{document}